\documentclass[11pt,a4paper,reqno]{amsart}
\usepackage{tikz,tikz-3dplot}
\usepackage{pgfplots}
\usepackage[margin=1in]{geometry} 
\usepackage{float}
\usepackage{amsfonts,amsmath,graphics,epsfig,latexsym,times}
\usepackage[format=hang, font=normalsize]{subfig}
\usepackage{ae,aecompl}
\newif\ifpix \pixtrue
\usepackage{setspace,fancyhdr,amssymb,amsthm,graphicx,ifthen,float,calrsfs} 
\numberwithin{equation}{section} \newtheorem{thm}{Theorem}[section]  
\newtheorem{cor}[thm]{Corollary}    
\newtheorem{lem}[thm]{Lemma}        
\newtheorem{prop}[thm]{Proposition}  
\theoremstyle{definition} \newtheorem{dfn}[thm]{Definition}
 
\newtheorem*{claim*}{Claim} 
 
\newtheorem{rmk}{Remark}

\allowdisplaybreaks[1]

\newcommand{\cA}{\mathcal{A}}
\newcommand{\tPhi}{\widetilde{\Phi}}
\newcommand{\tN}{\widetilde{N}}
\newcommand{\cR}{\mathcal{R}}
\newcommand{\cD}{\mathcal{D}}

\DeclareMathOperator{\Aut}{Aut} 
\DeclareMathOperator{\Av}{Av} 
\DeclareMathOperator{\pr}{pr}

\DeclareMathOperator{\diag}{diag} 
\newcommand{\rd}{{\mathrm d}}

\def\frt{{\mathfrak{t}}}

\DeclareMathOperator{\Div}{div}

\DeclareMathOperator{\Vol}{Vol}

\DeclareMathOperator{\supp}{supp} 
\DeclareMathOperator{\Int}{Int} 
\DeclareMathOperator{\Hess}{Hess}

\DeclareMathOperator{\Ker}{Ker}

\newcommand{\bC}{\mathbb{C}}
 
\newcommand{\bR}{\mathbb{R}} \newcommand{\bZ}{\mathbb{Z}}

 \newcommand{\bN}{\mathbb{N}}
 \newcommand{\bP}{\mathbb{P}}

\newcommand{\cI}{\mathcal{I}}

\newcommand{\cO}{\mathcal{O}} 

\renewcommand{\phi}{\varphi}
\renewcommand{\div}{\mbox{div}}
\newcommand{\ve}{\varepsilon}

\newcommand{\del}{\partial}

\renewcommand{\geq}{\geqslant}
\renewcommand{\leq}{\leqslant}
\renewcommand{\ge}{\geqslant}

\author{Florian T. Pokorny}
\address{Computer Vision and Active Perception Lab / Centre for Autonomous
Systems, School of Computer Science and Communication, KTH Royal
Institute of Technology, SE-100 44 Stockholm, Sweden}
\email{fpokorny@kth.se}
\author{Michael Singer}
\address{Department of Mathematics, University College London, WC1E 6BT}
\email{michael.singer@ucl.ac.uk}
\title{Toric partial density functions and stability of toric varieties
\footnotetext{\textsc{ Accepted by Mathematische Annalen on 13 September 2013.}}
}
\begin{document}

\begin{abstract} 
   Let $(L, h)\to (X, \omega)$ denote a polarized toric K\"ahler manifold. Fix a
toric submanifold $Y$ and denote by $\hat{\rho}_{tk}:X\to \bR$ the
partial density function corresponding to 
the partial Bergman kernel projecting smooth sections of $L^k$ onto holomorphic 
sections of $L^k$ that vanish to order at least $tk$ along $Y$, for
fixed $t>0$ such that $tk\in \bN$. We prove the existence of a 
distributional expansion of $\hat{\rho}_{tk}$ as $k\to \infty$, including the identification of the coefficient of
$k^{n-1}$ as a distribution on $X$.  This expansion is used to give a
direct proof that if $\omega$ has constant scalar curvature, then 
$(X, L)$ must be slope semi-stable with respect to $Y$
(cf. \cite{Thomas_cscK}).  Similar results are also obtained for more
general partial density functions.  These results have analogous
applications to the study of toric K-stability of toric varieties.
\end{abstract} 
\maketitle
\tableofcontents
\section{Introduction} 

\subsection{Polarized varieties and density functions}
Let $(X,L)$ be a smooth polarized projective variety of complex
dimension $n$.  Then the space $V_k = H^0(X,\cO(L^k))$ of holomorphic
sections of $L^k$ is a finite-dimensional vector space whose dimension
grows like $k^n$ as $k\to +\infty$.  The ampleness of $L$ also
corresponds to the existence of a metric $h$ on $L$ of positive
curvature, $F_h = -i\omega$, where $\omega$ is a K\"ahler form.
Denote by $g$ the associated Riemannian metric.   The choice of such a
metric $h$ with positive curvature equips $V_k$ with a
positive-definite innner product, making it a finite-dimensional
Hilbert space. 

The {\em density function} $\rho_k: M \to \bR$ associated with $(L,h)$ is defined by 
\begin{equation}\label{e1.15.8.12}
    \rho_k(p) = \sum_{\alpha} |e_{\alpha,k}(p)|^2, \quad \text{for }p\in M,
\end{equation}
where the $\{e_{\alpha,k}\}$ form an orthonormal basis of $V_k$.  Here
$|e_{\alpha,k}|^2$ is the point-wise length-squared of the the section
$e_{\alpha,k}$, computed using the metric $h^k$ on $L^k$.  We shall
refer to $|e_{\alpha,k}|^2$ as the {\em mass-density} of
$e_{\alpha,k}$. It is easy to see that $\rho_k$ is independent of the
choice of orthonormal basis $\{e_{\alpha,k}\}$.

The density function $\rho_k$ has been studied by many authors and it is known
that for large $k$, there is a complete asymptotic expansion
\begin{equation}\label{e1.18.8.11}
\rho_k \sim \left(\frac{k}{2\pi}\right)^n \sum_{j=0}^\infty a_j k^{-j},
\end{equation}
where the $a_j$ are smooth functions on $X$ which are local invariants of $g$,
so for example
\begin{equation}\label{e1a.18.8.11}
a_0 = 1 \mbox{ and }a_1 = \frac{1}{2}s,
\end{equation}
where $s$ is the scalar curvature of the K\"ahler metric $g$.
The reader is referred to the literature for the background to these
statements. Tian's famous paper \cite{Tian_1990} essentially gives the
$k^n$-term in this expansion; Zelditch \cite{Zelditch_1998} obtained the complete asymptotic
expansion \eqref{e1.18.8.11} (see also \cite{Catlin_1999},
\cite{MR2339952} and \cite{Berman_2008}). The identification of $a_1$
with half the scalar curvature appears in \cite{Lu_2000}.

The formula \eqref{e1.18.8.11} is to be viewed as a  local version of
the Hirzebruch--Riemann--Roch formula: indeed, integration over $X$
gives
\begin{equation}\label{e51.2.10.12}
\dim V_k \sim A_0 k^n + A_1 k^{n-1} + \cdots,
\end{equation}
where 
\begin{equation}\label{e52.2.10.12}
A_j = \int_X a_j\,\frac{\omega^n}{n!}.
\end{equation}
So
\begin{equation}\label{e53.2.10.12}
A_0 = \left(\frac{1}{2\pi}\right)^n \Vol_g(X) \mbox{ and }
A_1 = \left(\frac{1}{2\pi}\right)^n\int_X \frac{s}{2}\,\frac{\omega^n}{n!}.
\end{equation}
Note that \eqref{e51.2.10.12} also gives the leading coefficients of
the Hilbert polynomial $\chi(X,L^k)$ because the higher cohomology
groups vanish for large $k$.
The density function has been an important tool in the study of
K\"ahler metrics of constant scalar curvature over the last decade or so,
starting with Donaldson's pioneering work \cite{Don_Scal1} comparing
balanced metrics---which make $\rho_k$ constant for large $k$
---with metrics of constant scalar
curvature.   The reader is referred also to \cite{MR2161248,MR2508897,
  MR2296414,  MR2058389, MR2669363} for other contributions and
aspects of this circle of ideas.

In this paper we shall study a variant of $\rho_k$, the {\em partial
  density function} associated in the first instance to a complex
submanifold $Y$ of $X$.  Given a rational number $t\geq 0$, we may
twist $L^k$ by $\cI_Y^{tk}$, where $\cI_Y$ is the sheaf 
of functions vanishing along $Y$. This leads to a subspace
\begin{equation}\label{e11.12.9.12}
\hat{V}_{tk} = H^0(X, \cO(L^k)\otimes \cI_Y^{tk})
\end{equation}
of $V_k$ which of course inherits a Hilbert-space structure from
$V_k$.  Informally, $\hat{V}_{tk}$ is the space of sections of $L^k$
which vanish to  order at least $tk$ along $Y$.  We shall obtain a
distributional asymptotic expansion of $\hat{\rho}_{tk}$
 in the case that all data are {\em toric} and we shall use this 
expansion to give formula of the {\em slope} of the submanifold $Y$ 
in the sense of Ross and Thomas \cite{RT_07,Thomas_cscK}---see 
below for definitions.  From this we obtain an immediate proof 
that if the metric $g$ in the K\"ahler class $c_1(L)$ can be chosen to 
have constant scalar curvature, then $(X,L)$ must be 
slope semi-stable with respect to $Y$.

In order to state these results precisely, suppose that $X$ is a smooth toric variety with corresponding momentum
polytope $P \subset \bR^n$ and that $Y$ is an irreducible smooth toric subvariety
corresponding to a face $F$ of $P$. The reader is referred to
\S\ref{s_back} for a review of the correspondence between toric
varieties and convex polytopes. We may choose affine coordinates
$(x_j)$ on $\bR^n$ so that 
\begin{equation}\label{e1.12.9.12}
F = \{x_1 = \cdots = x_q = 0\} \cap P, 
\end{equation}
while $x_j\geq 0$ on $P$ for $j=1,\ldots, q$.  Let
\begin{equation}\label{e2.12.9.12}
\Phi(x) = x_1 + \cdots + x_q.
\end{equation}
Then $\Phi$ is non-negative on $P$ and vanishes precisely on $F$.  Pulled back to $X$,
$\Phi$ is everywhere non-negative, vanishing only on $Y$.  In fact,
$\Phi$ vanishes quadratically on $Y$: if local complex coordinates $z$ are chosen so that
$Y$ corresponds to $z_1=\cdots = z_q=0$, $\Phi = O(\sum_1^q |z_i|^2)$
(see Lemma~\ref{l1.17.11.11}).

Define three subsets of $X$:
\begin{equation}\label{three-subsets}
U_t = \Phi^{-1}[0,t), S_t = \Phi^{-1}(t), D_t = \Phi^{-1}(t,\infty),
\end{equation}
regarding $\Phi$ as a function on $X$.  Note that these subsets depend
upon the K\"ahler metric on $X$, not just the complex structure of $X$.

Denote by $\rd \sigma_t$ the Leray form of $S_t$: thus $\rd\sigma_t$
is a $(2n-1)$-form on $X$ satisfying $\rd \sigma_t\rd\Phi = \omega^n/n!$ along
$S_t$. Let $g$ be a toric K\"ahler metric on $X$ with scalar curvature
$s$.  Then we may define a distribution
$\hat{a}_t$ on $X$ (with support on $S_t$) by
\begin{equation}\label{e3.13.9.12}
\langle \hat{a}_t,f\rangle = \int_{S_t} f\,\rd \sigma_t -
\frac{1}{2}\frac{\rd}{\rd t}\int_{S_t} f|\rd\Phi|_g^2\,\rd \sigma_t
\mbox{ ($f\in C^{\infty}(X)$).}
\end{equation}

Then our first main theorem is as follows:
\begin{thm}\label{t1.5.9.12}
Let the notation be as above.  Then
\begin{equation}\label{e4.21.9.12}
\hat{\rho}_{tk}(p) = O(k^{-\infty})\mbox{ if }p\in U_t
\end{equation}
and
\begin{equation}\label{e5.21.9.12}
\hat{\rho}_{tk}(p) = \rho_k(p) + O(k^{-\infty})\mbox{ if }p\in D_t
\end{equation}
Moreover the $O$'s are uniform if $p$ moves in a compact subset
respectively of $U_t$ or $D_t$.

If $f \in C^{\infty}(X)$, then 
\begin{equation}\label{e54.5.9.12}
\langle \hat{\rho}_{tk},f\rangle =
\left(\frac{k}{2\pi}\right)^n\left( \int_{D_t} f \frac{\omega^n}{n!} +
  \frac{1}{2k}\left(\int_{D_t} sf \frac{\omega^n}{n!} + \langle
  \hat{a}_t,f\rangle\right)  + \frac{1}{k^{2}}\langle R_k, f\rangle\right).
\end{equation}
$R_k$ is a torus-invariant distribution on $X$ such that $\langle
R_k,f\rangle \leq C\|f\|_{C^{n+4}(X)}$, $C$ being bounded independent
of $k$. 
\end{thm}

\begin{rmk}
In fact, our methods give a complete distributional asymptotic
expansion of $\hat{\rho}_{tk}$, see Theorem~\ref{t1.5.10.12}.
\end{rmk}

\begin{rmk}
It is very natural to ask if Theorem~\ref{t1.5.9.12} can be extended
to the general case, i.e. when the data are not toric. By work of
Berman \cite{Berman_2009}, it is known in general that there will be open subsets $U_t$
and $D_t$ of $X$ (depending upon the K\"ahler structure) for which
\eqref{e4.21.9.12} and \eqref{e5.21.9.12} hold.  However, very little
is known about these sets: in particular, there is no reason to
believe $\partial U_t$ is smooth.  In the absence of some information
about the regularity of $\partial U_t$ it is hard to make a
conjecture about the asymptotic behaviour of
$\hat{\rho}_{tk}$ near $\partial U_t$ in the general case.
\end{rmk}

\subsection{Application to slope stability}
\label{s_slope}
The notion of slope stability of a polarized variety $(X,L)$ with
respect to a closed subscheme $Z$ is due to Ross and Thomas
\cite{Thomas_cscK, RT_07}.    It was introduced as part of
their study of K-stability of polarized varieties.  Its advantage is
that it may be relatively easy to show that $(X,L)$ is slope-unstable
with respect to a particular subscheme (or subvariety) implying
that $(X,L)$ cannot be K-stable.  K-stability is important
since it is conjectured to be the algebraic-geometric condition
which is necessary and sufficient for the existence of a K\"ahler
metric of constant scalar curvature in the K\"ahler class $c_1(L)$ -- at
least if $\Aut(X,L)$ is discrete 
\cite{MR1312688, MR1471884, Don_ToricScal}.  The necessity is known
\cite{SKD_Calabi, Stoppa_09} and a proof of the sufficiency has been
announced in the Fano case $L = - K_X$ \cite{CDS1,CDS2,CDS3,CDS4} and \cite{Tian_fano}. However, the
sufficiency for general cscK metrics (i.e., for general polarizations)
remains a major open problem.

One of our motivations for the study of the partial density function
$\hat{\rho}_{tk}$ was its application in the study of slope stability
proposed by Richard Thomas and his coworkers \cite{FKPST}. 
Theorem~\ref{thm_3} realizes this proposal by giving a formula for the
slope of a toric subvariety $Y$ in terms of geometric data defined by the
choice of a toric K\"ahler metric on $X$. From this formula, it is obvious that
if $g$ can be chosen to have constant scalar curvature, then $X$ is
slope (semi-)stable with respect to $Y$.

To describe these results more precisely, let us begin by recalling
the relevant definitions.
First of all, the slope $\mu(X) = \mu(X,L)$ of a polarized smooth
projective variety is defined 
in terms of the coefficients of the
Hilbert polynomial
\begin{equation}\label{e1.26.8.11}
\chi(X,L^k) = \sum_i (-1)^i \dim H^i(X,\cO(L^k)) = \dim V_k =  A_0k^n + A_1k^{n-1} + O(k^{n-2})
\end{equation}
as the quotient
\begin{equation}\label{e2.8.8.11}
\mu(X) = \mu(X,L) = \frac{A_1}{A_0}.
\end{equation}
Note that by \eqref{e53.2.10.12}, we have
\begin{equation}\label{e2a.8.8.11}
\mu(X,L) = \frac{1}{2}\Av(s),
\end{equation}
where $\Av(s)$ denotes the average scalar curvature on $X$.
Alternatively, we can define $\mu(X,L)$ by the formula
\begin{equation}\label{e2.26.8.11}
\dim H^0(X, \cO(L^k)) = A_0k^n( 1 + \mu(X,L)k^{-1} + O(k^{-2}))\mbox{
  for }k\gg 0.
\end{equation}

Now let $Z$ be a closed subscheme of $X$ with ideal sheaf $\cI_Z$. Then, if
$t \geq 0$ is such that $tk \in \bZ$, we may consider the holomorphic
Euler characteristic $\chi(X, L^k\otimes \cI_Z^{tk})$.  This will be a
polynomial of total degree $n$ in the two variables $k$ and $tk$ and
can therefore be written in the form
\begin{equation}\label{e3.26.8.11}
\chi(X,L^k\otimes \cI_Z^{tk}) = A_0(t)k^n + A_1(t)k^{n-1} + O(k^{n-2}),
\end{equation}
where $A_i(t)$ is a polynomial of degree at most $n-i$ (and so is defined for
all $t\in \bR$). As previously,
if $t$ is fixed and $k\gg 0$, the higher cohomology groups
vanish.  So we also have
\begin{equation}\label{e4.26.8.11}
\dim \hat{V}_{tk} = A_0(t)k^n + A_1(t)k^{n-1} +
O(k^{n-2})\mbox{ for fixed $t$, and  }k\gg 0,
\end{equation}
where we have written
\begin{equation}\label{e51.26.8.11}
\hat{V}_{tk} = H^0(X,\cO(L^k)\otimes \cI_Z^{tk}).
\end{equation}
Since $A_0(t)$ and $A_1(t)$ are defined for all real $t$, we may
consider the quantity
\begin{equation}\label{e5.26.8.11}
\mu_c(\cI_Z) = \mu_c(\cI_Z,L) = 
	\frac{\int_0^c A_1(t) + \frac{A_0'(t)}{2}\,\rd
        t}{\int_0^c A_0(t)\, \rd t},
\end{equation}
provided the denominator is non-zero: this is called the {\em slope} of $Z$
with respect to $c$.

We can guarantee that $A_0(t)>0$ for $t\in [0, \ve(Z))$, where
$\ve(Z)$ denotes the {\em Seshadri constant} of $Z$. Recall that one of the
equivalent definitions of this quantity is
\begin{equation}\label{e6.26.8.11}
\ve(Z) = \sup \left\{t:\pi^*L\otimes\cO(-t E) \text{ is ample}
	\right\},
\end{equation}
where $\pi:\widehat{X}\to X$ is the blow-up of $X$ along $Z$ and
$E=\pi^{-1}(Z)$ the exceptional divisor.

The {\em slope inequality} for $Z$ with respect to $c$ is:
\begin{equation}\label{e7.26.8.11}
\mu_c(\cI_Z,L) \leq \mu(X,L).
\end{equation}

Following \cite{Thomas_cscK}, one makes the following
\begin{dfn}\label{d1.17.11.11}
\begin{enumerate}

    \item[(i)] $(X,L)$ is said to be \emph{slope semi-stable with respect to $Z$}
  if the slope inequality \eqref{e7.26.8.11} holds for all $c \in
  (0,\ve(Z)]$. 

\item[(ii)] $(X,L)$ is said to be \emph{slope stable with respect to $Z$} if
  we have strict inequality in \eqref{e7.26.8.11} for all $c\in
  (0,\ve(Z))$, and for $c=\ve(Z)$ if $\ve(Z)$ is rational and 
  $H^0(X,\cO(L^k)\otimes \cI_Z^{\ve(Z)k})$ saturates $\cI_Z^{\ve(Z)k}$
  for $k\gg 0$.
\end{enumerate}
\end{dfn}

It is shown in \cite{RT_07} that, if $(X,L)$ is K-semistable, then
$(X,L)$ is slope-semistable with respect to every closed subscheme $Z$ and
that, if $(X,L)$ is (analytically) K-stable, then $(X,L)$ is
slope-stable with respect to every closed subscheme $Z$.  In the light of the
above-mentioned results (cscK implies K-stable, at least if $\Aut(X)$
is discrete) it follows that if there is a cscK metric in $c_1(L)$,
then every closed subscheme $Z$ of $X$ is slope stable. 

Theorem~\ref{t1.5.9.12} can be used to give a formula for
$\mu_c$. Note further that,
with this result, it is now easy to follow the approach of \cite{FKPST} to 
compute the difference $\mu_c(\cI_Y,L) - \mu(X,L)$:

\begin{thm}\label{thm_3}  Let the data be as in
Theorem~\ref{t1.5.9.12}.  Then we have, for $0 < c < \ve(Y)$:
\begin{equation}\label{e1.17.11.11}
\mu_c(\cI_Y,L) - \mu(X,L) = \left(2\int_0^c \Vol(P_t)\,\rd t\right)^{-1}
\left\{\int_0^c
  \left(\int_{U_t}(s-\Av(s))\frac{\omega^n}{n!}\right)\,\rd t
  -\frac{1}{2}\int_{S_c} |\rd \Phi|_g^2\,\rd\sigma_{c} \right\}.
\end{equation}
In particular, if the scalar curvature is constant ($s= \Av(s)$),
we have
\begin{equation}\label{e4.17.11.11}
\mu_c(\cI_Y,L) - \mu(X,L) <0, \mbox{ for }c\in (0,\ve(Y)),
\end{equation}
and $(X,L)$ is slope semi-stable with respect to $Y$.
\end{thm}

\begin{proof}
The function $1$ on $X$ is a test function and so we may pair the
distributional asymptotic expansion with it.  The integrals of the
local terms then give the coefficients $A_0(t)$ and $A_1(t)$:
\begin{equation}\label{e3.6.11.11}
A_0(t) = \left(\frac{1}{2\pi}\right)^n \Vol(U_t),\;\;
A_1(t) = \frac{1}{2}\left(\frac{1}{2\pi}\right)^n\left(\int_{U_t}s\, \frac{\omega^n}{n!} +
  \hat{A}_t\right),
\end{equation}
where we have written $\hat{A}_t$ for $\langle
\hat{a}_t,1\rangle$. Writing $s = (s-\Av(s)) + \Av(s)$ and
recalling \eqref{e2.8.8.11}, we obtain
\begin{equation}
A_1(t) = \left( \frac{1}{2\pi}\right)^n
\left( \Vol(U_t)\mu(X,L) + \frac{1}{2}\int_{U_t} (s-\Av(s))\frac{\omega^n}{n!} + \frac{1}{2}\hat{A}_t\right).
\end{equation}
Now
\begin{equation}
\hat{A}_t = \int_{S_t}\rd \sigma_{t} - \frac{1}{2}\frac{\rd}{\rd t} \int_{S_t}
|\rd \Phi|^2_g\,\rd \sigma_{t}
\end{equation}
and by definition of $\rd\sigma_t$, 
$$
\int_0^c\int_{S_t} \rd \sigma_{t} \rd t = \Vol(X) - \Vol(U_c).
$$
For $\epsilon>0$,
$$
\int_{\epsilon}^c \left( \frac{\rd}{\rd t}\int_{S_t}|\rd
  \Phi|^2\,\rd\sigma_{S_t}\right)\rd t = \int_{S_c} |\rd\Phi|_g^2\,\rd\sigma_{c} -
  \int_{S_{\epsilon}}     |\rd\Phi|_g^2\,\rd\sigma_{S_\epsilon}.
$$
Now $|\rd\Phi|^2\rd\sigma_t$ is smooth and tends to zero near
$Y$---its length with respect to $g$ is $|\rd \Phi|_g$ and $\Phi$ is
quadratic in the distance to $Y$.  It follows that
the integral over $S_{\epsilon}$  tends to zero with $\epsilon$ since $S_0 = Y$.  
The result follows by substitution of these formulae into
\eqref{e5.26.8.11}.  Indeed, the numerator of \eqref{e5.26.8.11} is
\begin{align}
\int_0^cA_1(t)\,\rd t + \frac{1}{2}(A_0(c) - A_0(0)) &=
\left( \frac{1}{2\pi}\right)^n\mu(X,L)\int_0^c \Vol(U_t)\,\rd t
\nonumber \\
&+ \frac{1}{2}\left( \frac{1}{2\pi}\right)^n\left(
\int_0^c\left(\int_{U_t} (s-\Av(s))\frac{\omega^n}{n!}
\right)\,\rd t - \frac{1}{2}\int_{S_c} |\rd\Phi|^2\,\rd \sigma_{c}\right).
\end{align}
Dividing by $\int_0^c A_0(t)\,\rd t$ as given in \eqref{e3.6.11.11}
and rearranging, we obtain the formula stated in the the theorem.
We obtain the slope semi-stability of $(X, L)$ with respect to $Y$ 
from the strict inequality \eqref{e4.17.11.11} by noting that
$c\mapsto\mu_c(\cI_Y, L)$ is continuous for $c\in (0, \ve(Y)]$.
\end{proof}

\subsection{More general partial density functions and K-stability}
\label{s_general}
There is a generalization of Theorem~\ref{t1.5.9.12} to partial density
functions associated with more general subspaces $\hat{V}_{tk}$ of
$V_k$.  We shall describe the set-up we have in mind in the toric
case.

Let $X$ correspond to the polytope $P$ as above, and suppose that we
have 1-parameter family of subpolytopes $P_t$ of $P$.  More
precisely, suppose that
\begin{equation}\label{e7.13.9.12}
P_t = P \cap \bigcap_{a\in A} \{\Phi_a(x) \geq t\}
\end{equation}
where the $\Phi_a$ are affine-linear functions on $\bR^n$ with
rational coefficients and $A$ is some finite index set.

We assume that $P_0 = P$ and are
interested in values of $t$ for which $P_t$ is an
$n$-dimensional polytope strictly smaller than $P$.

Because the coefficients of the $\Phi_a$ are rational, there is a
positive integer $N$ such that $kP_t$ is an integral polytope for all
positive integers $k$ divisible by $N$.  For such $k$, it therefore
makes sense to consider the subspace $\hat{V}_{tk}$ of holomorphic
sections of $L^k$ corresponding to the points of $(k^{-1}\bZ)^n \cap P_t$.
This subspace defines a partial density function $\hat{\rho}_{tk}$ as
before. Theorem~\ref{t1.9.9.12} below extends
Theorem~\ref{t1.5.9.12} to this more general setting.  We defer the
precise statement, but note here that the picture is similar to the
one we see in
Theorem~\ref{t1.5.9.12}: there is a tubular neighbourhood $D_t$ of a
reducible subvariety of $X$, with boundary $S_t$, such that
$\hat{\rho}_{tk}(x)$ is rapidly decreasing in $k$ if $x\in D_t$,
$\hat{\rho}_{tk}(x)$ is essentially equal to $\rho_k(x)$ if $x\not\in
\overline{D_t}$, and there is a distributional expansion of
$\hat{\rho}_{tk}$ with an explicit contribution $\hat{a}_t$ supported on $S_t$.
The set $S_t$ is now no longer smooth, however (it has singularities
at points mapping to intersection points of two or more of the
hyperplanes $\Phi_a(x)=t$) and there is an additional term in
$\hat{a}_t$ which is supported on the singular locus of $S_t$.

Just as Theorem~\ref{t1.5.9.12} gives information about
slope-stability if the metric is of constant scalar curvature, so
Theorem~\ref{t1.9.9.12} gives information about K-stability.  Indeed,
an analogue of Theorem~\ref{thm_3} is Theorem~\ref{t1.25.2.12} which
gives a formula for the Donaldson--Futaki invariant of a toric test
configuration in terms of the distributional expansion of
$\hat{\rho}_{tk}$.   This formula immediately implies that if the
metric has constant scalar curvature, then $(X,L)$ is K-polystable
with respect to every toric test configuration.  
 This result was previously
proved in \cite{MR2407096} without the use of density functions.

\begin{rmk}
Theorem~\ref{t1.9.9.12} can also be used to give a formula for
the slope $\mu_c(\cI_Z,L)$ for more general ideals $\cI_Z$.  We have
omitted an explicit treatment of this, however, because the
application to K-stability seems to be more interesting.
\end{rmk}
\subsection{Relation to previous work}
\label{relation}
Asymptotic
expansions of what we are calling partial density functions were
studied in detail by Shiffman and Zelditch \cite{MR2015330}.
Their point of view was that of random polynomials with prescribed
Newton polytope, and the partial density functions then appear as
`conditional expectations'.  Our results on the distributional
expansion of $\hat{\rho}_{tk}$ go beyond those of \cite{MR2015330}, 
by giving information about $\hat{\rho}_{tk}$ at
the interface $S_t$ between $U_t$ and $D_t$.  On the other hand, the
results of Shiffman and Zelditch in the interior of these regions are
much more precise than ours.   We mention also that these authors deal
only with the case that $X = \bC P^n$ with the Fubini-Study metric
(though the extension to general toric metrics is probably
straightforward) and
that their methods are completely different from ours, the starting
point being the description of the Szeg\"o kernel as a Fourier integral
operator with complex phase.  By contrast, our methods are
elementary and explicit.

Moving away from the toric case,  Berman \cite{Berman_2009}
announced that in general, given a complex submanifold $Y \subset X$
and with $\hat{V}_{tk}$ defined as in \eqref{e11.12.9.12}, there
exist open subsets $D_t$ and $U_t$ of $X$ satisfying the conditions
\eqref{e3.21.9.12} and \eqref{e4.21.9.12} of Theorem~\ref{t1.5.9.12}.
However, in this generality, there is no 
information about the smoothness of $\del D_t$ nor about the
`transition behaviour' of $\hat{\rho}_{tk}$ near $\del D_t$.

Finally we note that this work grew out of the first author's
Edinburgh PhD thesis \cite{Thesis_Florian} which contains further pointwise information
about the asymptotic expansion of toric partial density functions. 

\subsection{Outline}

The remainder of this paper is organised as follows. In
\S\ref{s_back}, we collect some standard notions from toric geometry.
The key to our subsequent analysis is the formula \eqref{e1.10.8.11},
which expresses the mass-density of a unit-norm basis element
$e_{\alpha,k}$ of $L^k$ in terms of a function $\phi(\alpha,y)$
derived in simple fashion from the symplectic potential $u$ which
defines our K\"ahler structure.

In \S\ref{setup}, we use Laplace's method to compute a distributional
asymptotic expansion of $|e_{\alpha,k}|^2$, following closely the
approach of \cite{Burns}. The key results here are
Propositions~\ref{p2.24.8.11}, \ref{p3.24.8.11}, and
\ref{p4.24.8.11}.  In \S\ref{sec_asymp}, these results are combined
with the Euler--Maclaurin formula for (lattice) polytopes to obtain
Theorem~\ref{t1.5.9.12}.

In \S\ref{Kstab} we shift attention to the more general partial
density functions mentioned in \S\ref{s_general}.  The method used to
obtain the distributional asymptotic expansion in this case is the
same as that followed in \S\ref{sec_asymp}:  it is, however, technically more complicated to
obtain a nice formula for the distributional term $\hat{a}_t$ in this more general case. The remainder of the
paper is devoted to the application of Theorem~\ref{t1.9.9.12} to
obtain a formula for the Donaldson--Futaki invariant of a toric test
configuration and to deduce that cscK implies K-polystable with
respect to such toric test configurations.

\subsection{Acknowledgement}
We thank  Julius Ross, Richard Thomas
and Steve Zelditch for useful conversations.   The second author was supported by a Leverhulme
Research Fellowship while this work was being completed.

\section{Background}
\label{s_back}
\label{sec_toric}
We review very briefly the elements we shall need of toric geometry,
referring the reader to \cite{Fulton_1993, Guillemin_1994,Abreu_1998,Abreu_2003,Burns} for more details.

\subsection{Combinatorial description of toric varieties}
\label{S_comb_toric}
First of all, we recall the correspondence between smooth projective
polarized toric varieties $(X,L)$ on the one hand and integral Delzant
polytopes on the other.  Thus we suppose that $X$ is a smooth (connected)
projective variety of complex dimension $n$ and that $X$
contains a dense open subset $X^o$ isomorphic to a complex $n$-torus
$T^n_c$.  We suppose further that the standard action of $T^n_c$ on
itself extends to a holomorphic action of $T^n_c$ on $X$.

Now suppose that $T^n \subset T^n_c$ is a compact real torus.  We are
interested in K\"ahler structures on $X$ that are invariant under
$T^n$, so that $T^n$ acts by isometries of the K\"ahler metric $g$.
As the action is holomorphic, the K\"ahler form $\omega$ is then
automatically $T^n$-invariant.

In this setting there is a moment map $\mu : X \to \frt^*$ (where $\frt^*$ is the dual
of the Lie algebra of $T^n$ which is isomorphic to $\bR^n$), and the image $P =
\mu(X)$ is a convex compact polytope, the convex hull of the images of
the fixed-points of the $T^n$-action. The restriction of $\mu$ to $X^o$ is
a fibration with image the interior $\Int(P)$ of $P$, with fibre
$T^n$.

The Lie algebra $\frt$ of $T$ contains the weight lattice $\Lambda =
\Ker(\exp)$ and correspondingly $\frt^*$ contains the coweight lattice
$\Lambda^*$. 

The condition that $X$ is smooth and has a given polarization $L$ translates into the condition that
$P$ is an {\em integral Delzant polytope}:
\begin{dfn}
    A convex polytope $P\subset \mathfrak{t}^*$ is \emph{Delzant} if
    \begin{enumerate}
        \item There are $n$ edges meeting in each vertex $v$.
        \item The edges meeting in the vertex $v$ are rational; i.e., each edge
          is of the form $v+te_i$, with $t\ge0$, $t\in \bR$ and $e_i\in \Lambda^*$.
        \item The $e_1,\dots,e_n$ in (2) can be chosen to form a basis of
          $\Lambda^*$.
    \end{enumerate}
    An \emph{integral Delzant polytope} in $\frt^*$ is a Delzant polytope whose
    vertices lie in $\Lambda^*$.
\end{dfn}

If $P$ is a Delzant polytope, we may write $P$ as the intersection $H_1
\cap \cdots \cap H_d$ of a finite number of affine half-spaces and we
may assume that for each $a$,
$$
H_a = \{ x \in \frt^* : \ell_a(x):= \langle x,\nu_a\rangle - \lambda_a \geq 0\},
$$
where $\nu_a\in \Lambda$ is primitive.  The intersection of the
boundary of $H_a$ with $P$ defines a codimension-1 face or {\em facet}
of $P$, which will be denoted $Q_a$:
\begin{equation}\label{e22.18.8.11}
Q_a = \{  x \in \frt^* : \ell_a(x) = 0\}\cap P.
\end{equation}
$P$ is integral if and only if all the $\lambda_a$ are integers.

More generally, if $Q$ is any face of codimension $q$ of $P$, there will be a
subset $\{a_1,\ldots, a_q\} \subset \{1,\ldots,d\}$ such that
\begin{equation}\label{e3.24.8.11}
Q = Q_{a_1}\cap \cdots \cap Q_{a_q}.
\end{equation}
Note that the conormal space $N^*Q$ (that is, the annihilator in $T^*\bR^n$
of $TQ$) is just the span of  $\{\nu_{a_1},\ldots,\nu_{a_q}\}$ (or
equivalently of the $\{\rd\ell_{a_1},\ldots \rd\ell_{a_q}\}$.

\subsubsection{Leray forms}
\label{s_leray}
\begin{dfn} Let $Q_a$ be a facet of $P$. The {\em Leray form} $\rd
  \sigma_a$ of $Q_a$ is the $(n-1)$-form on $Q_a$ with the property
  that $\rd \sigma_a\rd \ell_a = \rd x$ (Lebesgue measure) on $Q_a$. 
\label{e1.14.9.11}
Let $Q_{ab} = Q_a\cap Q_b$ be a codimension-2 face of $P$. The Leray
form $\rd \tau_{ab}$ of $Q_{ab}$ is similarly defined to be the
$(n-2)$-form on $Q_{ab}$ such that $\rd \tau_{ab}\rd \ell_a \rd \ell_b
= \rd x$.
\end{dfn}

In order to keep our notation short, we shall denote by $\rd \sigma$
the measure on $\partial P$ whose restriction to the relative interior of $Q_a$ is
$\rd\sigma_a$ and by $\nu$ the almost-everywhere defined section of
$T^*\bR^n$ such that $\nu$ is the conormal to $Q_a$ on its relative interior.
The measure $\rd \tau$ with support on the $(n-2)$-skeleton of
$\partial P$ is defined in the analogous way.

In \S\ref{Kstab}, we shall need to consider polytopes which are
not simple (so that more than $n$ facets can come together in a
vertex). Note that for any convex polytope, however, every codimension-2 face
is always the intersection of just 2 facets and so the Leray form
$\rd\tau$ is still well-defined in this case.  If the polytope is {\em
  simple}, then every face of codimension $q$ is the intersection of
precisely $q$ facets, and so has a well-defined Leray form.

\subsubsection{Adapted coordinates}

\begin{dfn}
If $p \in P$, there is a unique face $F$ of $P$ which contains $p$ in
its relative interior.  This {\em relative interior} will be denoted
by $F_p$.
\end{dfn}

In particular, $p\in F_p$.  The two extreme cases are $F_p = \Int(P)$
if $p$ is an interior point of $P$, and $F_p = \{p\}$ if $p$ is a
vertex of $P$.

If $p \in P$,  {\em adapted coordinates centred at $p$} will mean a
choice of affine coordinates $x$ on $\bR^n$ such that 
\begin{itemize}
\item $F_p$ is an open subset of $\{x_1 = x_2 = \cdots = x_q = 0\}$;
\item $x_j\geq 0$ on $P$ if $j=1,\ldots, q$;
\item the point $p$ corresponds to $x=0$;
\item Lebesgue measure on $\bR^n$ is given by $\rd x_1\cdots \rd x_n$.
\end{itemize}
It is clear that such coordinates always exist: if $F_p$ is the
relative interior of the face $Q$ in \eqref{e3.24.8.11}, then we take $x_j = \ell_{a_j}$ for $j=1,\ldots q$ and choose
the remaining coordinates so as to satisfy the remaining conditions.
Then the Leray form of $F_p$ is just $\rd x_{q+1}\cdots \rd x_{n}$.
In the case of a Delzant polytope, these coordinates can be chosen so
that $\Lambda^*$ is identified with the standard lattice $\bZ^n$, (in
other words so that the change of coordinates lies in
$SL(n,\bZ)$). The fact that $P$ can be covered by a finite system of adapted
coordinate charts will be useful in the next section.

Note that an adapted coordinate chart gives rise to a smooth system of
(local) coordinates on $X$ in the following way.  Letting
$(\theta_1,\ldots,\theta_n)$ be angular variables dual to the
coordinates $(x_1,\ldots, x_n)$ (i.e. the $\theta_j$ give coordinates
on $\frt$), then the real and imaginary parts of
$\sqrt{x_j}e^{i\theta_j}$, for $j=1,\ldots, q$ extend to be smooth
near $\mu^{-1}(F_p)$; and $(x_j,\theta_j)$ for $j=q+1,\ldots,n$ also
lift to be smooth functions near $\mu^{-1}(F_p)$.

Hence we have the following result:

\begin{lem}\label{l1.17.11.11}
Let $P$ be an integral Delzant polytope as before and let $\ell$ be
the defining function of a facet $Q$.  Then $\mu^*(\ell)$ is smooth on
$X$ and
vanishes quadratically on $Y= \mu^{-1}(Q)$ and is positive elsewhere
on $X$.

More generally, if a face $Q$ of $P$ is defined by the subspace $x_1=
\cdots = x_q=0$, the $x_j$ being $\geq 0$ elsewhere on $P$, then 
$\mu^*(x_1+\cdots + x_q)$ is smooth on $X$, vanishes quadratically on $\mu^{-1}(Q)$ and
is positive elsewhere on $X$.
\end{lem}

\subsubsection{Lattice points and holomorphic sections}

Two Delzant polytopes $P$ and $P'$ 
determine isomorphic toric varieties if they are combinatorially the
same and 
the set of normals to the facets of $P$ is the same as the set of
normals to the facets of $P'$. The polytope itself fixes in addition a
K\"ahler cohomology class $[\omega_P] \in H^2(X,\bR)$ which is integral if and only if the
polytope is integral\footnote{To be more precise, we should say
  that $[\omega_P]$ is integral iff $\mu$ (which is only determined up
  to an additive constant) can be chosen to make $P$ integral.}.  In this case, there is a $T^n_c$-invariant
holomorphic line bundle $L=L_P$ on $X$ such that $c_1(L) = [\omega_P]$ and
it is well known that the $T^n_c$-equivariant sections of $L$ are in
one-one correspondence with the points of $P\cap \Lambda^*$ and form a
basis of $H^0(X,\cO(L))$.  Replacing $L$ by $L^k$ corresponds to
replacing $\Lambda^*$ by the rescaled lattice
\begin{equation}
\Lambda^*_k = \{ y \in \frt^* : ky \in \Lambda^*\}
\end{equation}
so that there is a basis of sections $s_{\alpha,k}$ of
$H^0(X,\cO(L^k))$ indexed by $\alpha \in P\cap\Lambda^*_k$.
(Alternatively, we can think of the lattice as fixed and replace the
polytope $P$ by the dilated polytope $kP$ to get this basis of
sections.)

It is worth recalling that
\begin{equation}\label{e1.7.9.12}
s_{\alpha,k}(y) \neq 0 \Leftrightarrow \alpha \in \overline{F_{y}}
\end{equation}
and that the morphism $X \to \bP H^0(X,\cO(L^k))$ defined by this
basis of holomorphic sections is an embedding.

Let $F$ be a face of $P$, of codimension $q$. Then $Y=\mu^{-1}(F)$
is a toric subvariety of $X$, and conversely any irreducible toric subvariety of $X$ is equal to
$\mu^{-1}(F)$ for some face $Q$.  If $F$ is written as in
\eqref{e3.24.8.11}, then the normals $\nu_{a_1},\ldots, \nu_{a_q} \in
\Lambda$ generate a subtorus $T_F$ of $T^n$.  This is the stabilizer
of $Y$ which is toric with respect to the quotient torus $T^n/T_Q$.

\subsubsection{Seshadri constant}  If $Y$ is the subvariety
corresponding to the face $F$ of $P$, with $F$ defined as usual by the condition
$\Phi(x):=x_1+ \cdots + x_q=0$ in adapted coordinates, then from the definition \eqref{e6.26.8.11},
the Seshadri constant $\varepsilon(Y)$ of $Y$ is given by
\begin{equation}\label{e1.3.10.12}
\varepsilon(Y) = \sup\{t>0 : \Phi(x) = t\mbox{ contains no vertex of }P\}.
\end{equation}

\subsection{Toric K\"ahler metrics}
\label{s_met}
A choice of toric K\"ahler structure on $X$ adapted to the given
polarization $L$ corresponds to choosing a
{\em symplectic potential } $u$ on $P$ (see \cite{Abreu_2003}). Thus $u: P \to \bR$ is a
strictly convex function, smooth in the interior and satisfying the
boundary condition
\begin{equation}\label{e21.18.8.11}
u(x) - \frac{1}{2}\sum_{a=1}^d \ell_a(x)\log \ell_a(x) \in C^\infty(P)
\end{equation}
(i.e. this difference is smooth up to the boundary of $P$).
Given such a symplectic potential, set
\begin{equation}\label{e1.30.9.12}
H = \Hess(u), \mbox{ in other words } H_{ij} = \del_i\del_j u,
\end{equation}
and
\begin{equation}\label{e2.30.9.12}
G = H^{-1}.
\end{equation}
In terms of these matrices, the K\"ahler structure over
$X^o = \mu^{-1}(\Int(P))$ is given by
\begin{equation}\label{e2.7.9.12}
\omega = \rd x_j \wedge \rd \theta_j,\;\;
g = H_{ij}\rd x_i \rd x_j + G^{ij}\rd \theta_i \rd \theta_j,
\end{equation}
and the boundary condition \eqref{e21.18.8.11} ensures that this
extends smoothly to $X$.

Although $u$ is not smooth up to the boundary of $P$, the restriction
$u_F$ of $u$ to any face $F$ of $P$ is well-defined by
\eqref{e21.18.8.11}.  As part of the condition of `strict convexity', $u_F$ 
is required to be strictly convex and smooth in the interior of $F$
and to satisfy the analogous boundary conditions.  In fact, $u_F$ is
the symplectic potential for the restriction of the K\"ahler structure
to the toric submanifold $\mu^{-1}(F)$ of $X$.

The function 
\begin{equation}\label{e23.18.8.11}
u_0(x) = \frac{1}{2}\sum_{a=1}^d \ell_a(x)\log \ell_a(x)
\end{equation}
is strictly convex in $P$ and clearly satisfies
\eqref{e21.18.8.11}. This symplectic potential 
gives a special choice of toric K\"ahler structure on $X$ called the
Guillemin metric on $X$ \cite{Guillemin_1994}.

Note that the addition of an affine-linear function of $x$ to $u$ does
not affect the metric.  It does, however, affect the metric on the
line bundle whose curvature is the K\"ahler structure
\eqref{e2.7.9.12}.

\begin{dfn} \label{d1.13.9.12}
Denote by $s_{\alpha,k}$ a choice of section corresponding to the lattice
point $\alpha$, normalized so that the maximum value of
$|s_{\alpha,k}(y)|^2$ is equal to $1$.  
\end{dfn}
For each $\alpha$ and $k$, $s_{\alpha,k}$ is thus defined up to
multiplication by a unit complex number.

Following \cite{Burns}, define
\begin{equation}\label{e11.9.9.12}
\phi : P\times \Int(P) \longrightarrow \bR,\;\;
\phi(x,y) = 2(u(x) - u(y) - \langle \rd u(y),x-y\rangle).
\end{equation}

Then we have the key formula \cite{Burns, MR2394541}
\begin{equation}\label{e5.7.9.12}
|s_{\alpha,k}(y)|^2 = e^{-k\phi(\alpha,y)}
\end{equation}
for a section $s_{\alpha,k}$ normalized according to
Definition~\ref{d1.13.9.12}.

We note that for fixed $y$, $x\mapsto \phi(x,y)$ differs from $u(x)$
by an affine function of $x$.  In particular, it is strictly convex.
We also have
\begin{equation}\label{e1.13.9.12}
\phi(x,x) = 0,\; \nabla_x\phi(x,y)=0\mbox{ if }x=y\mbox{ and
}\nabla_y\phi(x,y) =0\mbox{ if }y=x.
\end{equation}
It follows from the convexity in $x$ that $\phi(x,y)\geq 0$ with
equality if and only if $x=y$, at least for $y\in \Int(P)$.

The $s_{\alpha,k}$ are automatically mutually orthogonal with respect
to the $L^2$ inner product, and so rescaling by the length of
$s_{\alpha,k}$ we obtain 
an $L^2$-orthonormal basis of sections $e_{\alpha,k}$, satisfying
\begin{equation}\label{e1.10.8.11}
|e_{\alpha,k}(y)|^2 = \frac{e^{-k\phi(\alpha,y)}}{(2\pi)^n\int_P
 e^{-k\phi(\alpha,z)}\,\rd z}.
\end{equation}
We shall refer to $|e_{\alpha,k}|^2$ as the {\em mass-density} of $e_{\alpha,k}$.

\begin{rmk}
The formula \eqref{e5.7.9.12} continues to be valid, with a suitable
extension of the definition of $\phi$, when $y\in \partial P$. This requires
some care: indeed, we see that for \eqref{e1.7.9.12} and 
\eqref{e5.7.9.12} to remain consistent,  we need
to define  $\phi(\alpha,y) = +\infty$ if $\alpha\not\in \overline{F_y}$.
\end{rmk}

\section{Asymptotic expansion of the mass-density function}
\label{setup}

The goal of this section is to obtain the large-$k$ asymptotic expansion of the quantity
\begin{equation}\label{e11.7.9.12}
\langle |e_{\alpha,k}|^2, f\rangle =
\frac{\int_P e^{-k\phi(\alpha,y)}f(y)\,\rd y}
{\int_P e^{-k\phi(\alpha,y)}\,\rd y},
\end{equation}
where $e_{\alpha,k}$ and $\phi$ are as in \eqref{e11.7.9.12} and \eqref{e1.10.8.11} and
$f$ is any smooth function on $P$.

We shall use Laplace's method for this, but this entails an
understanding of the critical points and some other global properties
of the function $y\mapsto \phi(\alpha,y)$ for fixed $\alpha$.  The
analysis is straightforward if $\alpha$ is an interior point of $P$
but a bit more complicated if $\alpha$ lies on the boundary.

We follow the argument of \cite{Burns} closely here.   We have
nonetheless provided the details, 
because their discussion applies only to the
Guillemin metric on $X$; and on the other hand Sena-Dias \cite{Dias}
provided the extension to general toric metrics but did not fully
analyze the situation at the boundary.

The following will be used in this section (and the rest of the paper):
\begin{itemize}
\item $H$  is the Hessian $\del_i\del_ju$ of $u$, and $G = H^{-1}$;
\item a Euclidean structure is fixed on $\bR^n$, the length of a
  vector $v$ being denoted by $|v|$;
\item we denote by $\|f\|_{r}$ the $C^r$-norm defined by our given
  Euclidean structure. 
\end{itemize}

\subsection{Properties of $u$ and $\phi$}

We begin with a statement of the properties of $H$ and $G$ that will
be needed later.

\begin{lem}\label{l1.28.8.12}
\begin{enumerate}
\item[(i)] There is a constant $c>0$ such that
\begin{equation}\label{e21.7.9.12}
\langle H(x)v,v\rangle \geq c|v|^2\mbox{ for all }x\in P, v\in \bR^n,
\end{equation}
where the LHS has to be interpreted as $+\infty$ if $v\not\in TF_x$;
\item[(ii)] $G = H^{-1}$ is smooth on $P$, $G(x)$ is
  positive-semidefinite for all $x\in P$ and
\begin{equation}\label{e22.7.9.12}
G(x)\xi = 0 \mbox{ if and only if }\xi \in N^*F_x
\end{equation}
(i.e. $\xi$ is conormal to $F_x$ at $x$.)
\end{enumerate}
\end{lem}
\begin{proof}
Let $p\in P$ and choose adapted coordinates such that $F_p$ is defined
by the vanishing of $x_1,\ldots, x_q$.  In particular these functions
are $\geq 0$ on $P$. It is convenient to write $x' = (x_1,\ldots,
x_q)$ and $x'' = (x_{q+1},\ldots, x_n)$.  

Set
\begin{equation}\label{e5.24.8.11}
D = \diag (x_1,\ldots, x_q).
\end{equation}
Then corresponding to the splitting of variables $x = (x',x'')$, we have the block decomposition
\begin{equation}\label{e3.13.9.11}
H = (u_{ij}) = \begin{pmatrix} (2D)^{-1} + H_0 & H_1 \cr H_1^t &
 H_2 \end{pmatrix}
\end{equation}
of the Hessian of $u$, where 
\begin{equation}
\begin{pmatrix} H_0 & H_1 \cr H_1^t & H_2 \end{pmatrix}
\end{equation}
is smooth.  At the boundary, $H_2$ is the Hessian of $u_{F_p}$, $H_2$
is positive-definite near $F_p$ (cf.\ \S\ref{s_met}). Hence $H$
is positive-definite and $\langle H(x)v,v\rangle = +\infty$ if and only if $v$ has
a non-zero component in the subspace spanned by $e_1,\ldots,e_q$, i.e if
$v$ is {\em not} tangent to $F_p$. Covering $P$ by a
finite number of open sets of this kind, a simple compactness argument
establishes part (i) of the lemma.

For part (ii), let 
$$
\Lambda = \begin{pmatrix} \sqrt{2}D^{1/2} & 0 \\ 0 &
  H_2^{-1/2} \end{pmatrix}
$$
Then
\begin{equation}\label{e12.4.9.12}
\Lambda H \Lambda = 1 +  R(D)
\end{equation}
where
\begin{equation} 
R(D) = 
\begin{pmatrix}
2D^{1/2}H_0D^{1/2} & \sqrt{2}D^{1/2} H_1 H_2^{-1/2} \\
\sqrt{2}H_2^{-1/2}H_1^tD^{1/2} & 0 \end{pmatrix}.
\end{equation}
Now certainly $\|R(D)\| = O( | x'|^{1/2})$ for small $x'$ and so
sufficiently close to $F_p$, we have
\begin{equation}
(1 + R(D))^{-1} =  \sum_{j=0}^\infty (-R(D))^j = 1 + \tilde{S}(D),
\end{equation}
say.  It is easy to see, moreover, that
\begin{equation}\label{e21.4.9.12}
\Lambda^{-1}\tilde{S}(D)\Lambda^{-1} = 
\begin{pmatrix} D & 0 \\ 0 & 1 \end{pmatrix}S(D)
\begin{pmatrix} D & 0 \\ 0 & 1 \end{pmatrix}
\end{equation}
where $S(D)$ is now a smooth function of $x'$.  It follows that the
inverse $G$ of $H$ has the form
\begin{equation}\label{e29.4.9.12}
\begin{pmatrix} 2D &  0 \\ 0 & H_2^{-1} \end{pmatrix} +
\begin{pmatrix} D & 0 \\ 0 & 1 \end{pmatrix}S(D)
\begin{pmatrix} D & 0 \\ 0 & 1 \end{pmatrix}.
\end{equation}
In particular $G$ is smooth up to
the boundary and everywhere positive-semidefinite.

For the last part, suppose first that $Q$ is a facet of $P$ and
suppose also that coordinates are chosen so that $Q = \{x_1=0\}$. Now
let $p\in Q$.  Then $F_p\subset Q$ and so $x_1$ will be among the
coordinates adapted to $F_p$ and centred at $p$.  With these choices,
if $\xi$ annihilates $TQ$ then it must be a multiple of $e_1$, and by \eqref{e29.4.9.12},
\begin{equation}\label{e31.4.9.12}
Ge_1 = 2x_1e_1 + x_1\begin{pmatrix} D & 0 \\ 0 & 1\end{pmatrix}
S(D)e_1,
\end{equation}
which shows that $Ge_1 = 0$ at $p$.  It follows that $Ge_1 = 0$ on the
whole of $Q$ (since $p\in Q$ was arbitrary).

If now $F = Q_1\cap\cdots \cap  Q_q$ is an arbitrary face of $P$, then choosing
adapted coordinates, we know that $Ge_j=0$ along $Q_j$, and so $G
e_j=0$ for all $j=1,\ldots q$ on $F$. Since $N^*F$ is the span of
$\{e_1,\ldots,e_q\}$, the `if' part of \eqref{e22.7.9.12} follows.

The `only if' part of \eqref{e22.7.9.12} is proved similarly.
\end{proof}

We now give some key properties of $\phi$.

\begin{lem}\label{l3.28.8.12}
\begin{enumerate}
\item[(i)]
The function $\phi(x,y)$ is smooth on $P \times \Int(P)$ and
there is a constant $c>0$ such that
\begin{equation}\label{e21a.7.9.12}
\phi(x,y) \geq c|x-y|^2 \mbox{ for all }x\in P,\; y \in \Int(P).
\end{equation}

\item[(ii)]  The function $\phi$ extends naturally to a function on
  $P\times P$ with   values in $[0,\infty]$ such that $\phi(x,y)= +
  \infty$ if and only if $x\not\in \overline{F}_y$ and satisfying
\begin{equation}\label{e21b.7.9.12}
\phi(x,y) \geq c|x-y|^2 \mbox{ for all }x,y\in P.
\end{equation}

\item[(iii)]  Let $p$ be a point of $P$ and let $x = (x',x'')$ be
  adapted coordinates centred at $p$.  Then there is a constant $C$ such that
  for all sufficiently small $y$,
\begin{equation}\label{e22a.7.9.12}
\phi(0,y) \leq \sum_{j=1}^q y_j + C|y|^2
\end{equation}
in these coordinates.
\end{enumerate}
\end{lem}
\begin{proof}
Suppose first that $y\in \Int(P)$.  Let $v\in \bR^n$ be any unit (with
respect to our arbitrary Euclidean structure) vector, and define
\begin{equation}\label{e11.13.9.12}
f(t) = \phi(y + tv,y).
\end{equation}
The domain of $f$ is the interval $I$ such that $y+tv \in P$. In
particular, $0\in I$ and by \eqref{e1.13.9.12}
\begin{equation}\label{e12.13.9.12}
f(0)=0,\;f'(0) = 0,
\end{equation}
and
\begin{equation}\label{e13.13.9.12}
f''(t) = \langle H(y+tv)v,v\rangle \geq c
\end{equation}
by Lemma~\ref{l1.28.8.12}.  Integrating this from $0$ to $t$ and using
\eqref{e12.13.9.12}, we obtain
$f(t) \geq ct^2/2.$  Since $|x-y| = |t|$, part (i) follows.

To understand the behaviour of $\phi$ near the boundary, let us write
\begin{equation}\label{e41.7.9.12}
u = u_0 + w,
\end{equation}
where $u_0$ is the Guillemin potential \eqref{e23.18.8.11} and $w$ is
smooth on $P$. A simple computation gives
\begin{equation}\label{e42.7.9.12}
\phi(x,y) = 
\sum \left(\ell_a(x)(\log\ell_a(x) - \log \ell_a(y)) - \ell_a(x-y)\right) +
\psi(x,y),
\end{equation}
where 
\begin{equation}
\psi(x,y) = 2[w(x) - w(y) - \langle \nabla w(y),x-y\rangle].
\end{equation}
is smooth, hence bounded, on $P\times P$.
If we fix $x$ and let $y \to y_0\in \partial P$, then it is clear that
$\phi(x,y) \to +\infty$ if there is an index $a$ with $\ell_a(y_0)= 0$ but $\ell_a(x) > 0$.  This is
precisely the condition $x \not\in \overline{F}_{y_0}$ which is consistent
with \eqref{e1.7.9.12} and \eqref{e5.7.9.12}.

It remains only to consider the situation that $x\in \overline{F_y}$
where $F_y$ is the interior of a proper face of $P$.  Now the
restriction $\phi_F$, say, of $\phi$ to $F$, is given by
\begin{equation}\label{e3.30.9.12}
\phi_F(x,y) = 2\left( u_F(x) - u_F(y) - \langle\nabla
  u_F(y),x-y\rangle\right),\;\; (x \in F, y \in \Int(F)).
\end{equation}
where $u_F$, the restriction of $u$ to $F$, is the symplectic potential
for the restriction of the K\"ahler structure to $\mu^{-1}(F)$.

Thus we can replace $P$ by $F$, $u$ by $u_F$ and $\phi$ by $\phi_F$ in
the argument at the beginning of this proof to obtain
\eqref{e21b.7.9.12} for $x\in \overline{F}_y$.  This completes the
proof of part (ii).

The last part is a local computation.  In adapted coordinates,
\begin{equation}
\phi(0,y) = 2[u(0) - u(y) + \langle \nabla u(y),y\rangle]
= \sum_{j=1}^q y_j - 2\{v(y) - v(0) - \langle \nabla v(y), y\rangle\}.
\end{equation}
Now the part in curly brackets is a smooth function of $y$ which vanishes
and has gradient $0$ at $y=0$.  Hence for sufficiently small $y$, we
can bound this by a multiple of $|y|^2$, giving
$$
\phi(0,y) \leq \sum_{j=1}^q y_j + C|y|^2
$$
as required.
\end{proof}

\subsection{Distributional asymptotic expansion of $|e_{\alpha,k}|^2$}

The main goal of this section is the following:

\begin{prop}
Let $f$ be a smooth $T^n$-invariant function on $X$ and denote by the
same letter the corresponding function on $P$.  Denote by $s$ the
scalar curvature of the metric with symplectic potential $u$.  For
each $\alpha \in P\cap \Lambda^*_k$, recall that $e_{\alpha,k}$ is the
unit-length holomorphic section of $L^k$ corresponding to the point
$\alpha$. Then we have
\begin{equation}\label{ep2.24.8.11}
\langle |e_{\alpha,k}|^2,f\rangle = f(\alpha)+
\frac{1}{2k}\left(s(\alpha)f(\alpha) + \frac{1}{2}\del_i\del_j(G^{ij}f)(\alpha)\right) +
\frac{1}{k^2}\langle R_k(\alpha),f\rangle
\end{equation}
where $R_k(\alpha)$ is a distribution which satisfies
\begin{itemize}
\item for fixed $\alpha\in P$, $\langle R_k(\alpha),f\rangle \leq C\|f\|_{C^4}$;
\item for each fixed test-function $f$, $\langle R_k(\alpha),f\rangle$
  is smooth in $\alpha$ and bounded for $k\gg 0$.
\end{itemize}
\label{p2.24.8.11}\end{prop}

Here we recall that $G = (G^{ij})$ is the inverse of the Hessian of the symplectic
potential and that $s$ is the scalar curvature of the metric $g$.  We
recall also Abreu's famous formula for the scalar curvature of the
metric \eqref{e2.7.9.12}
\begin{equation}\label{e1.4.9.13}
s = - \frac{1}{2}\del_i\del_j G^{ij} \mbox{ (summation convention)}.
\end{equation}

We begin with a stronger result covering the case that $\supp(f)$ does
not contain $\alpha$.

\begin{prop}
Suppose that $\alpha\in P$ and $f\in C^\infty(P)$ with $\alpha\not\in
\supp(f)$. Then
\begin{equation}\label{e3.10.8.11}
\langle |e_{\alpha,k}|^2,f \rangle = O(k^{-\infty})
\end{equation}
for large $k$. 
\label{p3.24.8.11}\end{prop}
\begin{proof}  Since
\begin{equation}
\langle |e_{\alpha,k}|^2,f \rangle = 
\frac{\int_P   e^{-k\phi(\alpha,y)}
f(y)\,\rd y}{\int_P
  e^{-k\phi(\alpha,y)}\,\rd y},
\end{equation}
we need an $O(k^{-\infty})$ upper bound for the numerator and an
$O(k^N)$ lower bound for the denominator.   In fact we shall obtain an
exponentially small upper bound for the denominator. 

By Lemma~\ref{l3.28.8.12},
\begin{equation}
|e^{-k\phi(\alpha,z)}f(z)| \leq e^{-ck|\alpha-z|^2}\supp |f|
\end{equation}
and so if the distance from $\alpha$ to $\supp(f)$ is $d$,
\begin{equation}
\left|\int e^{-k\phi(\alpha,y)}f(y)\,\rd y\right | \leq  \|f\|_0\int_{|x-y|\geq d}
  e^{-ck|x-y|^2}\,\rd y
\end{equation}
Now 
\begin{equation}
\int_{|z| \geq d} e^{-kc|z|^2} \,\rd z \leq Ce^{-kd^2}
\end{equation}
for some constant $C$ independent of $k$.  Hence
\begin{equation}\label{e41.5.9.12}
\left|\int e^{-k\phi(\alpha,z)}f(z)\,\rd z\right| \leq Ce^{-kcd^2}\|f\|_0.
\end{equation}
We complete the proof by obtaining a suitable lower bound on the total
mass of $s_{\alpha,k}$. Suppose that $y = (y',y'')$ are adapted
coordinates centred at $\alpha$ so that we are in the situation of
part (iii) of Lemma~\ref{l3.28.8.12}.  Suppose further that the subset
\begin{equation}\label{e31.13.9.12}
V = \{0 \leq y_j \leq \epsilon\mbox{ for }j=1,\ldots,q\} \times \{  |y''| \leq \epsilon\}
\end{equation}
is contained in $P$.  Now by \eqref{e22a.7.9.12},
\begin{equation}
\phi(0,y) \leq \sum_{j=1}^q y_j + C|y'|^2 + C|y''|^2
\end{equation}
and by shrinking $\epsilon$ if necessary, we may absorb the
$|y'|^2$-term into the linear term, getting
\begin{equation}\label{e1.28.8.12}
\phi(0,y) \leq  2\sum_{j=1}^q y_j + C|y''|^2 
\end{equation} 
for $y \in V$.
Then
\begin{equation}\label{e2.28.8.12}
\int_P e^{-k\phi(\alpha,y)}\,\rd y \geq \int_V
e^{-k\phi(\alpha,y)}\,\rd y 
\geq
\int_V \exp\left( - k\left(2\sum_{j=1}^q y_j +
    C|y''|^2\right)\right)\,\rd y.
\end{equation}
Now the difference between this integral and the integral over
$\bR^q_+\times \bR^{n-q}$ is 
exponentially small in $k$, with a constant depending upon
$\epsilon$).  Since
\begin{equation}\label{e3.28.8.12}
\int_0^\infty  e^{-2ky} = \frac{1}{2k}\mbox{ and }
\int_{-\infty}^\infty e^{-kCy^2} = \frac{\sqrt{\pi}}{\sqrt{kC}},
\end{equation}
it follows that
\begin{equation}\label{e5.28.8.12}
\int_P e^{-k\phi(\alpha,y)}\,\rd y \geq Ck^{-q - (n-q)/2} = Ck^{-(n+q)/2}.
\end{equation}
Dividing \eqref{e41.5.9.12} by \eqref{e5.28.8.12} completes the proof.
\end{proof}

The effect of this Proposition is to localize this
$\langle |e_{\alpha,k}|^2,f\rangle$ to an integral over an arbitrarily
small neighbourhood of $\alpha$ in $P$, up to exponentially small terms.
We now calculate this contribution recursively.

With $\alpha$ fixed as before, choose adapted coordinates $y = (y',
y'')$ as in the previous proof, and let $V$ be as in \eqref{e31.13.9.12}.
Choose a cut-off function $\rho\in C^{\infty}_0(V)$, $0\leq \rho\leq
1$, where $\rho=1$ in a smaller neighbourhood $W = \frac{1}{2} V$ of
$\alpha$ in $P$.

Define the operator $\delta : C_0^\infty(V) \to C^\infty(V,\bR^n)$
\begin{equation}\label{e6.28.8.12}
\delta_j f (y) = \int_0^1 \del_j f(ty)\,\rd t
\end{equation}
so that
\begin{equation}\label{e11.15.8.11}
f(y) - f(0) = y_j\delta_jf(y)
\end{equation}
(summation convention) for all $y\in V$. For any function $f\in C^\infty_0(V)$ 
define the linear operator $\cD: C^\infty_0(V) \to C^\infty_0(V)$ by
\begin{equation}
\cD f = \frac{1}{2}\del_i(\rho G^{ij} \delta_j f).
\end{equation}
Note that the operator $\cD$ depends also on the point $\alpha$.  When
we need to draw attention to this fact, we shall denote it also by
$\cD_\alpha$. 

The significance of this operator is as follows
\begin{prop}
With the notation as above, we have, for any $N \geq 1$,
\begin{equation}\label{e21.15.8.11}
\langle |e_{\alpha,k}|^2,f\rangle 
= \sum_{m=0}^N k^{-m}\cD_\alpha^mf(\alpha) +
  k^{-N-1}\cR_{N+1,k,\alpha}(f),
\end{equation}
where the remainder term $\cR_{N+1,k,\alpha}(f)$ is smooth in $\alpha$
for fixed $f$ and satisfies
\begin{equation}\label{e32.13.9.12}
\cR_{N+1,k,\alpha}(f) \leq C_N\|f\|_{C^{2N+2}}.
\end{equation}
uniformly in $\alpha$ and $k$.
\label{p4.24.8.11}\end{prop}

\begin{proof}
Given the test-function $f$, write
\begin{equation}\label{e11.30.9.12}
f(y) =  \rho(y)f(0) + \rho(y)(f(y)-f(0)) + (1- \rho(y))f(y),
\end{equation}
and substitute this into $\int e^{-k\phi(\alpha,y)}f(y)\,\rd y$,
getting
\begin{align}
\int_P e^{-k\phi(y)}f(y)\rd y &=
f(0)\int_P
e^{-k\phi(y)}\rho(y)\,\rd y + \int_P
e^{-k\phi(y)}\rho(y)(f(y)-f(0))\,\rd y \\
&+ \int_P e^{-k\phi(y)}(1-\rho(y))f(y)\,\rd y.
\end{align}
In the second term, use \eqref{e11.15.8.11} and note also that
\begin{equation}
\del_j e^{-k\phi(y)} = -2kH_{ij}y_ie^{-k\phi(y)},\mbox{ so that }
\frac{1}{2}G^{ij}\del_j e^{-k\phi(y)} = -ky_ie^{-k\phi(y)}.
\end{equation}
Hence
\begin{equation}
\int_P e^{-k\phi(y)}\rho(y)(f(y)-f(0))\,\rd y = -k^{-1}\int_P \frac{1}{2}G^{ij}\del_i
e^{-k\phi}\rho\delta_j f\rd y = k^{-1}\int_P e^{-k\phi} \cD f\rd y,
\end{equation}
where we have neglected the boundary term
\begin{equation}
\int_{\partial P} G^{ij}\nu_j e^{-k\phi}\rho\delta_if\,\rd \sigma.
\end{equation}
This is justified because $G^{ij}\nu_j = 0$ on the interior of each
facet of $P$ (see part (ii) of Lemma~\ref{l1.28.8.12}).

In summary, then, we have the formula
\begin{align}\label{e12.30.9.12}
\int_P e^{-k\phi(\alpha,y)}f(y)\rd y &= f(\alpha) \int
e^{-k\phi(\alpha,y)}\rho(y)\,\rd y + k^{-1} \int
  e^{-k\phi(\alpha,y)}\cD_\alpha f(y)\,\rd y  \nonumber \\
&+ \int e^{-k\phi(\alpha,y)}(1-\rho(y))\,\rd y.
\end{align}
We can now iterate: we apply \eqref{e12.30.9.12} to the second term on
the right-hand side, (i.e. with $f(y)$ replaced by $\cD_\alpha
f(y)$).  After $N$ steps, we obtain the formula
\begin{align}\label{e13.30.9.12}
\int_P e^{-k\phi(\alpha,y)}f(y)\,\rd y &=
\cA_{N,\alpha}(f)(\alpha)\int e^{-k\phi(\alpha,y)}\rho(y)\,\rd y
+
k^{-N-1}\int_P e^{-k\phi(\alpha,y)}\cD^{N+1}_{\alpha}f(y)\,\rd y \nonumber \\
&+
\int_P e^{-k\phi(\alpha,y)}(1-\rho(y))\cA_{N,\alpha}f(y)\,\rd y.
\end{align}
From the proof of Proposition~\ref{p3.24.8.11}, we have
\begin{equation}\label{e1.4.10.12}
\int_P e^{-k\phi(\alpha,y)}(1-\rho(y))\,\rd y =
e^{-ck\epsilon}\eta_k(\alpha)
\end{equation}
for some $c>0$ where $\eta_k(\alpha)$ is smooth in $\alpha$ and
uniformly bounded in $k$ provided that $\alpha$ moves in some smaller
subset $\frac{1}{4}V$, say. Moreover, \eqref{e2.28.8.12} and
\eqref{e3.28.8.12} imply that
\begin{equation}\label{e2.4.10.12}
\left(\int_p e^{-k\phi(\alpha,y)}\,\rd y\right)^{-1}\int_P
e^{-k\phi(\alpha,y)}\,\rd y = 1 + e^{-ck\epsilon}\eta_k'(\alpha),
\end{equation}
where $\eta'_k$ has the same properties as $\eta_k$.

Hence, dividing by $\int e^{-k\phi}$, we get \eqref{e21.15.8.11}, where
\begin{align}\label{e15.30.9.12}
\cR_{N+1,k,\alpha}(f)  &=
k^{N+1}\eta_k(\alpha)e^{-ck\epsilon}\cA_{N,\alpha}f(\alpha) \nonumber
\\
&+ 
\left(\int_p e^{-k\phi(\alpha,y)}\right)^{-1}
\int e^{-k\phi(\alpha,y)}\left\{\cD^{N+1}_\alpha f(y)
+ k^{N+1}(1-\rho(y))\cA_{N,\alpha}f(y)
\right\}
\,\rd y.
\end{align}
It is clear from this formula that for fixed $f$ and $k$,
$\cR_{N+1,k,\alpha}$ is a distribution in $f$ depending smoothly on
$\alpha$. 

On the other hand, by \eqref{e5.28.8.12} and \eqref{e1.4.10.12} we
have 
\begin{equation}\label{e16.30.9.12}
|\cR_{N+1,k,\alpha}(f)| \leq C(\sup|\cD^{N+1}_\alpha f| + e^{-ck\epsilon} \sup|\cA_{N,\alpha}f|)
\end{equation}
directly from \eqref{e15.30.9.12}.
Since the operator $\delta$ has the same boundedness properties as a differential operator,
\begin{equation}
    \|\delta f \|_r \leq A\| f \|_{C^{r+1}}
\end{equation}
for $r\geq 0$, where $A=A_r$ is some constant, it follows that the operator $\cD$ behaves like a
second-order operator in the sense that we have an estimate:
\begin{equation}
    \|\cD f\|_r \leq A \|f \|_{C^{r+2}}
\end{equation}
(for some different constant $A = A_r$).  It follows by induction that
$\sup| \cD_\alpha^m f|$ is bounded by a multiple of
$\|f\|_{C^{2m+2}}$. The estimate
\begin{equation}\label{e17.30.9.12}
\|\cR_{N+1,k,\alpha}f| \leq C\|f\|_{C^{2N+2}}
\end{equation}
now follows by combining these observations with \eqref{e16.30.9.12}.
\end{proof}

To obtain Proposition~\ref{p2.24.8.11} from this expansion, we take
$N=1$, getting
\begin{equation}
\langle |e_{\alpha,k}|^2,f \rangle = 
f(\alpha)  + k^{-1}\cD f(\alpha) + k^{-2}\cR_{2}(f),
\end{equation}
and it follows from the formula for $\cR_2$ that this error term has
the stated properties.  It remains to compute $\cD f(\alpha)$.   In
local coordinates, with $\alpha$ corresponding to $0$ as before, 
\begin{equation}
\delta_jf(y) = \del_j f(0) + \frac{1}{2}\del_i\del_j f(0)y^i +
O(|y|^2)
\end{equation}
from the Taylor expansion of $f(y) - f(0)$ and, after a little
manipulation, we obtain
\begin{equation}
\cD f(0) = \frac{1}{4}\del_i\del_j(f G^{ij})(0) -
\frac{1}{4}f(0)\del_i\del_j G^{ij}(0).
\end{equation}
The formula \eqref{ep2.24.8.11} now follows from Abreu's formula \eqref{e1.4.9.13}.

\section{Proof of Theorem~\ref{t1.5.9.12}}

\label{sec_asymp}
We now bring the ideas of the previous sections together to prove
Theorem~\ref{t1.5.9.12}.  Recall that the setting for that Theorem was
as follows:
\begin{itemize}
\item A toric variety $X$ with moment polytope $P$;
\item A face $F = Q_1\cap  \cdots Q_q$ with $Q_j$ defined by $x_j=0$
  for $j=1,\ldots,q$.
\item The subpolytope $P_t = P \cap \{\Phi(x)\geq t\}$, where $\Phi(x)
  = x_1+\cdots+x_q$.
\end{itemize}
Then our partial density function is given by
\begin{equation} \label{e31.26.8.11}
\hat{\rho}_{tk}(y) = \sum_{\alpha\in P_t\cap\Lambda^*_k}
|e_{\alpha,k}(y)|^2
\end{equation}
(regarded, by abuse of notation, as a function of $y\in
P$), where the terms in the sum are given by \eqref{e1.10.8.11}.

Define 
\begin{equation}
C_t = \Phi^{-1}[0,t), N_t = \Phi^{-1}(t), P_t = \Phi^{-1}(t,\infty).
\end{equation}
These are the subsets of $P$ corresponding respectively to the three
subsets $U_t$, $S_t$ and $D_t$ in \eqref{three-subsets}.  By
torus-invariance, it is clearly enough to prove the `pushed-down'
version of Theorem~\ref{t1.5.9.12}, i.e. to work entirely on $P$.

We begin by establishing the first part of the
Theorem~\ref{t1.5.9.12}, namely the equations \eqref{e4.21.9.12} and
\eqref{e5.21.9.12} restated as follows:

\begin{prop}\label{p1.1.10.12}
Let $K$ be any compact subset
of $C_t$. Then
\begin{equation}\label{e41.4.9.12}
\hat{\rho}_{tk}(x) = O(k^{-\infty})\mbox{ uniformly for }x\in K
\end{equation}
and if $K'$ is a compact subset of $P_t$, then
\begin{equation}\label{e41a.4.9.12}
\hat{\rho}_{tk}(x) = \rho(x) + O(k^{-\infty})\mbox{ uniformly for }x\in K'
\end{equation}
\end{prop}

\begin{proof}
If $x\in K$ and $\alpha \in P_t$, we have
\begin{equation}
|e_{\alpha,k}|^2 \leq  Ce^{-kd(K,P_t)},
\end{equation}
where $d$ denotes Euclidean distance.  Summing over lattice points of
$D_t$ gives the result, since the number of lattice points is
$O(k^n)$.  The proof of the other part is the same, the roles of $C_t$
and $P_t$ being interchanged.
\end{proof}
\begin{rmk}
  If $X = \bC  P^n$ with the Fubini--Study metric, then more precise
  pointwise estimates of this kind are given in \cite{MR2015330}, at
  least for points $x$ in the interior of $P$.  There, $D_l$ is called
  the `forbidden region'.
\end{rmk}

\subsection{The Euler--Maclaurin formula}

In order to obtain an expansion in powers of $k$ from
\eqref{e31.26.8.11}, we use the Euler--Maclaurin formula to replace
the sum over lattice points by an integral, up to a controlled error
term.  The version we use is as follows:
\begin{thm}\label{EMthrm}
Let $P$ be a convex integral polytope of dimension $n$, with integral conormals.  Let Lebesgue measure $\rd x$
be normalized so that the integral of the unit cube in $\bZ^n$ has
volume $1$ and let $\rd\sigma$ stand for the Leray form of $\del P$.
Then we have
\begin{equation}\label{e51.5.9.12}
\sum_{P\cap\Lambda_k} f(\alpha) = k^n\int_{P} f(x)\,\rd x +
\frac{k^{n-1}}{2}\int_{\partial P} f\,\rd \sigma + k^{n-2}E_k(P,f)
\end{equation}
where $E_k(P,f)$ is bounded by a multiple of $\Vol(P)\|f\|_{2n}$ (the
$C^{2n}$-norm of $f$ again).
\end{thm}
\begin{rmk}  Since many results of this kind are available in the
  recent literature, we shall be content to sketch a proof. Following
  the method used by Donaldson in the appendix of 
  \cite{Don_ToricScal}, we reduce to the case that $P$ is a lattice
  simplex. Then we are content to quote the Euler--Maclaurin formula
with remainder from \cite{Karshon_2003} to complete the proof.

We note references such as \cite{Guillemin_Polytopes,
  Karshon_2003}) give complete asymptotic expansions of lattice sums
at least if $P$ is a simple polytope.  The theorem stated here applies
to any lattice polytope, and this will be important in \S\ref{Kstab}.
This simple statement \eqref{e51.5.9.12} should be viewed as an
extension of the `trapezium rule' (with remainder) for approximate
integration of functions of one variable.
\end{rmk}

\begin{proof}
Given $f\in C^\infty(P)$, consider on the one hand
$$
S_k(f,P) = \sum_{\alpha \in P\cap \Lambda_k} f(\alpha) - \frac{1}{2}\sum_{\alpha
  \in \del P\cap \Lambda_k} f(\alpha).
$$
and 
$$
I(f,P) = \int_P f(x)\,\rd x.
$$
We aim to show first that
\begin{equation}\label{e21.30.9.12}
S_k(f,P) = k^nI(f,P) + O(k^{n-2}),
\end{equation}
where the $O(k^{n-2})$ error term stands for a distribution supported
on $P$ and bounded by a multiple of $\|f\|_{C^{2n}}$.  For this, note
first that if $P$ is decomposed as a union of polytopes $P_1$ and $P_2$
with disjoint interiors, then 
$$
S_k(f,P) = S_k(f,P_1) + S_k(f,P_2) + O(k^{n-2})
$$
because the number of points of $\Lambda_k$ where there is a discrepancy is contained
in the $(n-2)$-skeleton of $P_1\cap P_2$ and hence bounded by a
multiple of $k^{n-2}$.  (The $O(k^{n-2})$ error is also bounded by a
multiple of $\sup_P(f)$.) In this situation we also have
$$
I(f,P) = I(f,P_1) + I(f,P_2).
$$
From these considerations, since we can decompose our polytope into
integral simplices with disjoint interiors, it is enough to establish
\eqref{e21.30.9.12} for integral simplices.  Although it is not hard to
prove this by induction, we may simply invoke, for example, Theorem~1
of \cite{Karshon_2003} which, after rescaling, gives
\begin{equation}
S_k(f,\Sigma) = k^n I(f,\Sigma) + k^{n-2}E_k(\Sigma,f)
\end{equation}
for any integer simplex $\Sigma$, where $E_k$ is a distribution on
$\Sigma$ which is bounded by a multiple of $\|f\|_{2n}$.

This is not quite the result we need, but if $F$ is any facet of $P$,
then we have
\begin{equation}
S_k(f,F) = \sum_{\alpha\in F\cap \Lambda_k} f(\alpha) + O(k^{n-2})
\end{equation}
(again because there are only $O(k^{n-2})$ points in the
$(n-2)$-skeleton of $F$) and by what we've just proved,
\begin{equation}
S_k(f,F) = k^{n-1}I(f,F) + O(k^{n-3}).
\end{equation}
Combining these observations with\eqref{e21.30.9.12},
we see that we can replace the sum over lattice points of the boundary
by the corresponding integral, up to an allowable error term.

This completes our sketch proof.
\end{proof}

\subsection{Divergence theorem}

Apart from the Euler--Maclaurin formula, we also need a formula for
the integral of the divergence of a vector field over the intersection
of a hyperplane with $P$.  In fact it is natural to consider a
one-parameter family of parallel hyperplanes
\begin{equation}\label{e1.21.9.12}
W(t) = \{ \Phi(x) = t\},
\end{equation}
where $\Phi$ is an affine-linear function on $\bR^n$.  In this
situation we make the following definition:
\begin{dfn} The number $c\in \bR$ is called a critical value of the
  one-parameter family $P\cap W(t)$ if $W(c)$ contains a vertex of $P$.  If
  $c$ is not a critical value of $P\cap W(t)$, we call it a regular
  value of $P\cap W(t)$.
\label{d1.28.9.12}
\end{dfn}

We note that if $t_0$ is not a critical value of $P\cap W(t)$, then for
sufficiently small $\delta>0$, $P\cap W(t)$ and $P\cap W(s)$ are
combinatorially identical and have the same conormals for $s,t \in
(t_0-\delta,t_0+\delta)$. It follows that if $f$ is a smooth function
on $P$, then
\begin{equation}
t \mapsto \int_{P\cap W(t)} f\,\rd \sigma_t
\end{equation}
(which is continuous for all $t$) is smooth for $t \in
(t_0-\delta,t_0+\delta)$.  Here $\rd \sigma_t$ is the Leray form of
$P\cap W(t)$, i.e.
$$
\rd \sigma_t \,\rd \Phi = \rd x
$$
along $W(t)$, where $\rd x$ is short-hand for the standard euclidean
measure on $\bR^n$.

\begin{lem}\label{l11.28.8.12}
Let $P$ be a convex polytope in $\bR^n$ and $W(t)$ be as
above, and suppose that $t_0$ is a regular value of this one-parameter
family. Let $\xi$ be a smooth vector field on $P$.  Let $P(t)$ be the
part of $P$ cut off by the half-space $\{\Phi(x)\geq t\}$.

Denote by $\rd \sigma$ the Leray form of the codimension-1 part of the
boundary of $P(t)$ and by $\rd \tau$ the Leray form of the codimension-2 part of
the boundary.
Then for all $t$ in a sufficiently small neighbourhood of $t_0$,
\begin{align}\label{e8.2.11.11}
	\int_{W(t)} \div(\xi)\,\rd \sigma = 
\frac{\rd}{\rd t} \int_{W(t)} \langle \xi,\rd\Phi\rangle \,\rd
  \sigma  - \int_{\partial W(t)} \langle
  \xi,\nu\rangle \rd\tau.
\end{align}
\end{lem}

\begin{proof} Fix $t$ and $t+h$ near $t_0$ so that the interval
  $[t,t+h]$ contains no critical value (we assume $h>0$ here). Define
  the polytope $C(h)$ to be the closure of $P(t)\setminus P(t+h)$.
  The facets of $C(h)$ are the two parallel facets $P\cap W(t)$ and
  $P\cap W(t+h)$ together with $\{ C(h)\cap G\}$, where $G$ is a facet
  of $P$.  Denote by $Z(h)$ the union of these `side' facets of
  $C(h)$.  Applying the divergence theorem to $C(h)$, we have
\begin{equation}\label{e9.2.11.11}
\int_{C(h)} \div(\xi)\,\rd x = 
- \int_{P\cap W(t)}\langle \xi,\rd\Phi\rangle \,\rd \sigma_t +
\int_{P\cap W(t+h)}\langle \xi,\rd\Phi\rangle \,\rd \sigma_{t+h}
- \int_{Z(h)} \langle \xi,\nu\rangle\,\rd \sigma.
\end{equation}
We will now calculate the limit as $h\to 0$ of this equation.

By definition of Leray form, for any smooth function on $P$,
\begin{equation}\label{e2.21.9.12}
\int_{C(h)} f\,\rd x = \int_{t}^{t+h} \left( \int_{P\cap W(s)} f\,\rd
    \sigma_s\right)\,\rd s,
\end{equation}

Thus if $h$ is
small, we have
\begin{equation}\label{e3.21.9.12}
\int_{C(h)} f\,\rd x = h \int_{P\cap W(t)} f\,\rd \sigma_t  + O(h^2).
\end{equation}
Similarly, for each facet $G$ of $P$ meeting $W(t)$, we have
\begin{equation}\label{e4a.21.9.12}
\int_{G\cap C(h)} f\,\rd \sigma = h \int_{G\cap W(t)} f\,\rd \tau_t  + O(h^2).
\end{equation}
so that
\begin{equation}\label{e4c.21.9.12}
\int_{Z(h)} f\,\rd \sigma = h \int_{\partial(P\cap W(t))} f\,\rd \tau_t  + O(h^2).
\end{equation}

Thus the LHS of
\eqref{e9.2.11.11} is equal to
\begin{equation}\label{e15.5.9.12}
h\int_{W}\div(\xi)\,\rd\sigma +O(h^2)
\end{equation}
while the first two terms on the RHS combine to give
\begin{equation}\label{e16.5.9.12}
h\,\frac{\rd}{\rd t}\int_{W(t)}\langle \xi,\nu\rangle \,\rd
  \sigma + O(h^2).
\end{equation}
Finally the integral over $Z(h)$ is
\begin{equation}\label{e17.5.9.12}
h\int_{\partial{W}}\langle \xi,\nu\rangle\,\rd \tau + O(h^2).
\end{equation}
Combining these three equations, dividing by $h$, and taking $h$ to $0$
thus gives \eqref{e8.2.11.11} as required.
\end{proof}

With these preliminaries we can now establish Theorem~\ref{t1.5.9.12}

\subsection{Completion of Proof of Theorem~\ref{t1.5.9.12}}

We now complete the proof of Theorem~\ref{t1.5.9.12} by deriving the
distributional formula \eqref{e54.5.9.12}.  Note first that that
formula is written on $X$ rather than downstairs on $P$. It is clear
that $\hat{\rho}_{tk}$ is $T^n$-invariant, so it is enough to obtain
\eqref{e54.5.9.12} for functions $f$ of the form
$\mu^{*}(\overline{f})$, where $f\in C^\infty(P)$.  Since the volume
of each fibre of $\mu$ is $(2\pi)^n$, 
\begin{equation}
\int_X \mu^*(f)\omega^n/n! = (2\pi)^n\int_P f\,\rd x.
\end{equation}
Thus, identifying $\hat{\rho}_{tk}$ and $\hat{a}_{tk}$  with their
respective push-downs to $P$, we see
that \eqref{e54.5.9.12} is equivalent to the formula
\begin{equation}\label{e1.5.10.12}
\langle \hat{\rho}_{tk},f\rangle =
k^n\left( \int_{P_t} f + \frac{1}{2k}\left(\int_{P_t} sf\,\rd x + \langle
  \hat{a}_t,f\rangle\right)  + \frac{1}{k^{2}}\langle R_k,
f\rangle\right)\;\; (f\in C^\infty(P))
\end{equation}
where $R_k$ is an appropriate error term and
\begin{equation}\label{e2.5.10.12}
\langle \hat{a}_t , f\rangle = \int_{N_t} f\,\rd \sigma  -
\frac{1}{2}\frac{\rd}{\rd t}\int_{N_t} f|\rd\Phi|^2_g\,\rd \sigma.
\end{equation}
In the remainder of this section we shall always think of
$\hat{\rho}_{tk}$ and $\hat{a}_{tk}$ as distributions on $P$ rather
than on $X$.

With these preliminaries understood, we just combine the
Euler--Maclaurin formula \eqref{e51.5.9.12} with the distributional
expansion \eqref{ep2.24.8.11}, getting
\begin{equation}\label{e35.5.9.12}
\langle \hat{\rho}_{tk},f\rangle = k^n\int_{D_t} f
+ \frac{1}{2}k^{n-1}\left(\int_{\partial P_t} f\,\rd\sigma  +
  \int_{P_t} (s - \frac{1}{2}\partial_i\partial_j(G^{ij}f))\right) + O(k^{n-2}).
\end{equation}

Now, by the divergence theorem,
\begin{equation}\label{e25.4.9.12}
\int_{P_t} \partial_i\partial_j(G^{ij}f) = - \int_{\partial
 P_t} \partial_j(G^{ij}f)\nu_i\,\rd \sigma = 
- \int_{N_t}
\partial_j(G^{ij}f)\nu_i\,\rd \sigma 
- \int_{\del
P^+_t} \partial_j(G^{ij}f)\nu_i\,\rd \sigma,
\end{equation}
where we have written $P_t^+ = P_t\setminus N_t$ for the `old' part of
the boundary of $P_t$.
Consider the second term on the RHS, and more specifically a facet $F$
of $P$.  We may suppose that $F$ is given by $x_1=0$, so the conormal
is $e_1$ and $\rd \sigma = \rd x_2\ldots \rd x_n$.

From \eqref{e31.4.9.12},
\begin{equation}\label{e31a.4.9.12}
Ge_1 = 2x_1e_1 + x_1\begin{pmatrix} D & 0 \\ 0 & 1\end{pmatrix}
S(D)e_1 = (2x_1 + O(x_1^2))e_1 + x_1\eta(x),
\end{equation}
say, where the vector $\eta(x)$ is orthogonal to $e_1$. Hence
$$
fG\nu = f(2x_1 + O(x_1^2))e_1 + x_1f\eta(x)
$$
and so
$$
\Div(fG\nu) = 2f + O(x_1).
$$
By a similar argument to that used to prove \eqref{e22.7.9.12}, this is true
uniformly up to the boundary of $F$ and so the contribution from this
facet to the integral over $\partial P_t^+$ is $2\int_F f\,\rd
\sigma$.  Hence \eqref{e25.4.9.12} simplifies to
\begin{equation}\label{e26.4.9.12}
\int_{P_t} \partial_i\partial_j(G^{ij}f) = - 
\int_{N_t} \partial_j(G^{ij}f)\nu_i\,\rd \sigma 
- 2\int_{\del
P^+_t} f\,\rd \sigma.
\end{equation}

Now use Lemma~\ref{l11.28.8.12} on the integral over $N_t$ to get
\begin{equation}\label{e2.4.9.12}
\int_{N_t} \partial_j(G^{ij}f)\nu_i\,\rd \sigma 
= \frac{\rd }{\rd t}
\int_{N_t} fG^{ij}\nu_i\nu_j\,\rd \sigma_t,
\end{equation}
the contribution from the boundary of $N_t$ being zero.  The reason
for this is as follows.  Each boundary facet of $N_t$ is of the form
$N_t \cap F$, where $F$ is a facet of $P$. By the lemma, the
integrand will be $fG^{ij}\nu_i\nu'_j$, where $\nu'$ is the conormal
to $F$. But we have seen that $G^{ij}\nu'_j=0$ on $F$ in the previous
part. Hence the boundary contribution is zero.  Combining these
calculations, we arrive at
\begin{equation}\label{e61.4.9.12}
\int_{P_t} \del_i\del_j(fG^{ij})\,\rd x =
-\frac{\rd}{\rd t}\int_{N_t} f|\rd\Phi|^2_g\,\rd \sigma_t -
2\int_{\partial P_t^+} f\,\rd \sigma.
\end{equation}

\label{s_complete}
Combining equations \eqref{e35.5.9.12} and \eqref{e61.4.9.12}, we
obtain
\begin{eqnarray}
\langle \hat{\rho}_{tk},f\rangle &=&  k^n\int_{P_t} f\,\rd x +
\frac{1}{2}k^{n-1}\int_{P_t} fs\,\rd x \nonumber \\
&+& \frac{1}{2}k^{n-1}\left(\int_{\partial P_t} f\,\rd \sigma -
 \int_{\partial P_t^+} f\,\rd \sigma - \frac{1}{2}\frac{\rd}{\rd
   t}\int_{N_t} f|\rd\Phi|^2_g\,\rd \sigma \right) +
O(k^{n-2}). \label{e36.5.9.12} 
\end{eqnarray}
Now the first two terms in the middle line combine to give
$\int_{N_t}f\,\rd \sigma$, completing the proof of
\eqref{e1.5.10.12}.  The bound 
$$
\langle R_k,f\rangle \leq C\|f\|_{C^{n+4}}
$$
follows from the bounds on the error terms in \eqref{ep2.24.8.11} and \eqref{e51.5.9.12}.

\subsection{Complete asymptotic expansion}

We note that these methods yield a complete distributional asymptotic
expansion for $\hat{\rho}_{tk}$:
\begin{thm}\label{t1.5.10.12}
There exists a sequence $\xi_j$ of distributions on $P$ such that for each
$N>0$, we have
\begin{equation}\label{e5.5.10.12}
\langle \hat{\rho}_{tk},f\rangle = k^n\left\{ \sum_{j=0}^N\langle \xi_j,
f\rangle k^{-j} + k^{-N-1}\langle R_{N+1,k},f\rangle\right\},
\end{equation}
where $R_{N+1,k}$ is a distribution on $P$ satisfying
$$
\langle R_{N+1,k},f\rangle \leq C\|f\|_{C^{N'}}
$$
for some $N'$ depending upon $N$.
\end{thm}

\begin{proof}
We sketch the proof as we will not use the result in the rest of the
paper.  Fix a test-function $f$ and an integer $N>0$.
Proposition~\ref{p4.24.8.11} gave an asymptotic expansion of $\langle
|e_{\alpha,k}|^2,f\rangle$ to order $N$, and we took care to note that
all the coefficients as well as the error term depend smoothly upon
$\alpha$.  On the other hand, from the results of for example
\cite{Karshon_2003} or \cite{Guillemin_Polytopes}, for any smooth
function $u(\alpha)$ on $P_t$, we have an asymptotic expansion of the
lattice sum
\begin{equation}\label{e7.5.10.12}
\sum_{P_t\cap\Lambda_k^*} u(\alpha) = k^n\sum_{j=0}^N I_j(u)k^{-j} +
k^{-N-1}\langle S_{N+1,k},u\rangle,
\end{equation}
where the $I_j$ and
$S_{N+1,k}$ are certain distributions on $P$ with the error term
$S_{N+1,k}$ satisfying
$$
\langle S_{N+1,k},u\rangle \leq C\|f\|_{C^{N'}}
$$
for some integer $N'$ depending upon $N$.  Because all the
coefficients in our expansion of $\langle |e_{\alpha,k}|^2,f\rangle$
are smooth in $\alpha$ , we may substitute \eqref{e21.15.8.11}
into \eqref{e7.5.10.12}, getting an asymptotic expansion of the form \eqref{e5.5.10.12}
\end{proof}

\subsection{Combinatorial interpretation}

Finally, we give a combinatorial interpretation of the
sub-leading term if $f=1$:
\begin{prop}\label{p1.9.9.12}
If $t>0$ is a regular value of the family $P_t$, then we have
\begin{equation}\label{e1.9.9.12}
\Vol(\partial P_t^+)  = \int_{D_t}s\,\rd x - \frac{1}{2}\frac{\rd}{\rd
  t}\int_{N_t} f|\rd\Phi|^2_g\,\rd \sigma_t.
\end{equation}
\end{prop}
\begin{proof}
If we plug $f=1$ into the above formula, we know that we get the
leading coefficients $A_0(t)$ and $A_1(t)$ as the leading
coefficients.  On the other hand, the dimension of the space of
sections is the number of lattice points in $P_t$ and this is
approximated by the Euler--Maclaurin formula with $f=1$.  Comparing coefficients now
gives the result, at least if $t$ is rational. Since both sides are
continuous in $t$, the result is true for all regular values $t$.
\end{proof}

\section{More general partial density functions and toric K-stability}
\label{Kstab}
In this section we want to generalize Theorem~\ref{t1.5.9.12} to more general
subspaces of $V_k$, defined by a general rational convex polytope
${P}_t$ of $P$, obtaining in particular a distributional asymptotic
expansion for the partial density function $\hat{\rho}_{tk}$ defined
in this situation.

We shall then use this distributional expansion to prove that toric
cscK implies toric K-stable in a sense to be explained below.  These
results appear as Theorems~\ref{t1.9.9.12} and \ref{t1.25.2.12}.

We begin with a careful discussion of what we shall call `polytopes
with moving facets'.  We have already seen the simplest example where
a 1-parameter family of polytopes ${P}_t$ is defined by
intersecting a given polytope $P$ with a variable half-space
$\{\Phi(x)\geq t\}$. This is a polytope with a single moving facet.
We must generalize this to allow for an intersection of $P$ with an
arbitrary finite collection of half-spaces $\{\Phi_a(x)\geq t\}$ where each
of the $\Phi_a$ is an affine function of $x$.  

\subsection{Polytopes with moving facets}
\label{s_mov}
Let $P\subset \bR^n$ be a convex integral polytope. Suppose
given a finite collection $\Phi_a(x)$, $(a\in A)$ of affine-linear functions, with
rational coefficients. For each $t\in \bR$, define
\begin{equation}\label{e1.21a.9.12}
P(t) = P \cap \bigcap_{a} \{\Phi_a(x)\geq t\}.
\end{equation}
We assume that ${P}(0) = P$. For any given $t$, define $A(t) \subset A$ to be
the subset of `effective constraints', i.e. $a\in A_t$ if $P(t)\cap
\{\Phi_a(x)=t\}$ is a facet of $P$.  With the assumption $P(0) = P$ in
force, it follows that $A(0) = \emptyset$. 

\vspace{11pt}

\noindent{\bf Notational Remark}: In this section we have started to
denote the $t$-dependent polytope as $P(t)$ rather $P_t$, and
similarly for other $t$-dependent quantities that in the previous
section were denoted by a subscript $t$. It is hoped that this will
make this section more readable.

\vspace{11pt}

Because the $\Phi_a$ are rational, given a positive rational number
$t$, there is a positive integer $N$ such that for all integers $k$
divisible by $N$, $P(t)$ is a lattice polytope for the rescaled lattice
$\Lambda^*_k$ (equivalently $kP(t)$ is a lattice polytope for
$\Lambda^*$ for such $k$).  

\begin{dfn}\label{d2.1.10.12}
For rational $t>0$ and integers $k>0$ as in the previous paragraph,
define $\hat{V}_{tk} \subset V_k$ to be the span of the sections of
$L^k$ corresponding to points in $P(t)\cap \Lambda_k^*$.  Similarly,
given a choice of toric metric $h$ on $L$, 
$\hat{\rho}_{tk}$ is defined to be the partial density function for
the subspace $\hat{V}_{tk}$.
\end{dfn}

In addition to the
combinatorial data $P$ and $P(t)$, we now choose a toric metric $g$ on $X$.
In order to state our generalization of Theorem~\ref{t1.5.9.12}, we
need the following notation and definitions.  
\begin{itemize}
\item Write $\partial P(t) = N(t) \cup \partial P^+(t)$, where
  $N(t)$---the `new part' of the boundary---is the union of those
  facets defined 
  by $\Phi_a(x)=t$ for $a\in A(t)$.  The Leray form is denoted, as
  usual, by $\rd \sigma$.
\item A positive measure $\rd p$ supported on the $(n-2)$-skeleton of
  $N(t)$ is defined as follows: for any pair of facets $N_a(t) =
  \{\Phi_a(x)=t\}$ and $N_b(t) = \{\Phi_b(x)=t\}$, define 
\begin{equation}\label{e11.5.10.12}
\rd p_{ab} = |\rd\Phi_a - \rd\Phi_b|^2_g\,\rd \tau_{ab}
\end{equation}
on $N_a(t)\cap N_b(t)$, and define $\rd p$ to be equal to $\rd p_{ab}$
on the relative interior of $N_a(t)\cap N_b(t)$.

\item The notion of a `regular value' and `critical value' of the
  family $P(t)$ defined below.
\end{itemize}
For the last of these, note that the set $A(t)$ is locally constant in
$t$ in general, but will jump at a finite number of values of
$t$. These will be called the critical values of the family and will
be denoted $c_j$:
\begin{equation}\label{e12.5.10.12}
0 = c_0 < c_1 < \cdots < c_m.
\end{equation}
By definition $P(t) = \emptyset$ for $t<0$ and $t> c_m$.  A value of
$t$ not equal to one of the $c_j$ will be called `regular'.   The
significance of this notion is that if $t,s\in (c_{j-1},c_j)$ for some
$j$, then the facets of $P(t)$ and $P(s)$ have the same conormals and
are combinatorially identical. In particular, the Leray forms 
depend smoothly upon $t$ for $t$ in any one of these intervals.

Given these preliminaries, we can state our generalization of
Theorem~\ref{t1.5.9.12} as follows.

\begin{thm}\label{t1.9.9.12}  Let $P$ and $P(t)$, and $\hat{\rho}_{tk}$
  be defined as above. Let
\begin{equation}\label{e1.1.10.12}
C(t) = P\setminus P(t),
\end{equation}
so that $P$ is decomposed into mutually disjoint subsets $C(t)$, $N(t)$
and $D(t) = P(t)\setminus N(t)$.

Then the partial density function $\hat{\rho}_{tk}$ associated to
$P(t)$ has the following properties:
\begin{equation} \label{e3.1.10.12}
\hat{\rho}_{tk}(x) = O(k^{-\infty})\mbox{ if }x\in C(t)
\end{equation}
and
\begin{equation}\label{e4.1.10.12}
\hat{\rho}_{tk}(x) = \rho_k(x) + O(k^{-\infty})\mbox{ if }x\in D(t).
\end{equation}
Moreover, the $O$'s are uniform if $x$ moves in a compact subset
respectively of $C(t)$ or $D(t)$.

Let $f \in C^{\infty}(P)$. Then, provided that $t$ is a regular value of
the family $P(t)$, 
\begin{equation}\label{e54a.5.9.12}
\langle \hat{\rho}_{tk},f\rangle =
k^n\left( \int_{D(t)} f + \frac{1}{2k}\left(\int_{D(t)} sf + \langle
  \hat{a}_t,f\rangle\right)  + O\left(\frac{1}{k^2}\right)\right),
\end{equation}
where
\begin{equation}
\langle \hat{a}_{tk},f\rangle =
\int_{N(t)} f\,\rd \sigma - \frac{1}{2}\frac{\rd}{\rd t}\int_{N(t)}
f|\rd \Phi|^2_g\,\rd \sigma - \int_{N(t)} f\,\rd p
\end{equation}
and where $O(1/k^2)$ denotes a distribution $R_k$ such that $\langle
R_k,f\rangle \leq Ck^{-2}\|f\|_{C^{n+4}}$.
\end{thm}

\begin{proof}
Equations \eqref{e3.1.10.12} and \eqref{e4.1.10.12} are established
following exactly the same argument as for
their counterparts in Proposition~\ref{p1.1.10.12}.

Moreover, the strategy for obtaining \eqref{e54a.5.9.12} is exactly the same as
for \eqref{e1.5.10.12}: the only difference is that the calculation of
\begin{equation}
\label{71.4.9.12}
\int_{P_t} \del_i\del_j(G^{ij}f)\,\rd x
\end{equation}
is more complicated.  

Indeed, the first step is the same and the analogue of
\eqref{e26.4.9.12} here is
\begin{equation}
\label{e72.4.9.12}
\int_{P(t)} \del_i\del_j(G^{ij}f)\,\rd x
= - \int_{N(t)} \del_i(f G^{ij}\del_j\Phi)\,\rd \sigma - 2\int_{\del
  P^+_{t} } f\,\rd \sigma.
\end{equation}
On the right-hand side, we have used an obvious shorthand: the first
term should more properly be written as
\begin{equation}\label{e21.1.10.12}
- \sum_{a\in A_t} \int_{N_a(t)} \del_i(fG^{ij}\del_j\Phi_a)\,\rd
\sigma_a.
\end{equation}
Consider a typical term 
\begin{equation}\label{e1.7.10.12}
\int_{N_a(t)} \del_i(fG^{ij}\del_j\Phi_a)\,\rd\sigma_a
\end{equation}
in this sum. We want to use
Lemma~\ref{l11.28.8.12} to simplify this integral.  For this, let $s>t$
and consider $P(t)\setminus P(s)$, which we think of as a
neighbourhood of $N(t)$. Decompose this set as a union of polytopes
$C_a(t,s)$; $C_a(t,s)$ is defined to be the convex hull of $N_a(t)$
and $N_a(s)$.  Because $t$ is a regular point of the family $P_t$,
$N_a(t)$ and $N_a(s)$ are combinatorially identical and their boundary
facets have the same conormals for all $s$ sufficiently close to
$t$.  See Figure~1 for an illustration of this
construction: $N(t)$ is $BCDEF$, $N(s)$ is $HJILM$; and the
$C_a(t,s)$ here are $BCJHB$, $CDIJC$, $DELID$ and $EFMLE$.

\begin{figure}[!ht]
    \vspace{0.5cm}
{
    \includegraphics[width=75mm]{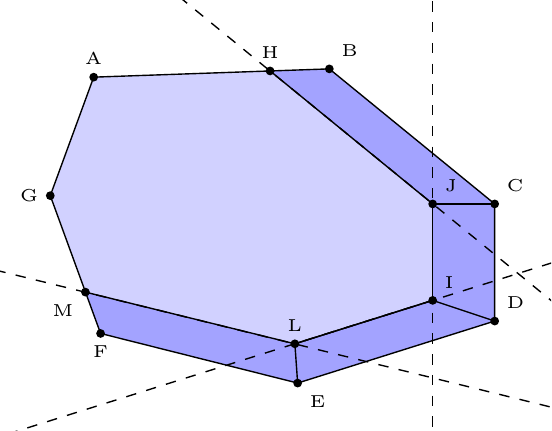} 
\label{figures0}
}
\caption{}
\end{figure}

 We claim that we can apply Lemma~\ref{l11.28.8.12} with $P$
replaced by $C_a(t,s)$ and $W(t)$ replaced by $N_a(t)$. For this to be
the case, we need to know that the `side faces' of $C_a(t,s)$ do not
vary as $t$ is varied. Now a typical side facet is either the
intersection of $C_a(t,s)$ with an old boundary facet $F$ of $\partial
P^+$ or else the intersection $Z_{ab}(t,s)$ with $C_b(t,s)$ for some
$b\neq a$. It is clear that $F$ does not move with $t$.  As for
$Z_{ab}(t,s)$, any point $x$ on it satisfies $\Phi_a(x) = \Phi_b(x) = t'$,
for $t\leq t'\leq s$. In particular, the hyperplane containing the facet $Z_{ab}(t, s)$ is given by $\Phi_b(x) -
\Phi_a(x)=0$ and its inward conormal is $\rd\Phi_b-\rd \Phi_a$.
So such a facet also does not move with $t$. 

Thus we can apply Lemma~\ref{l11.28.8.12} to obtain
\begin{align}\label{e1.6.9.12}
    \int_{N_a(t)} \div(\xi)\,\rd \sigma_a =&
\frac{\rd}{\rd t}\int_{N_a(t)} \langle \xi,\rd \Phi_a\rangle\,\rd
\sigma_a
- \int_{\partial N_a(t)\cap \partial P_t^+} \langle
\xi,\nu\rangle\,\rd \tau\\
&- \sum_b \int_{N_a(t)\cap N_b(t)} \langle
\xi,\rd(\Phi_b-\rd\Phi_a\rangle\,\rd \tau_{ab}
\end{align}
for any smooth vector field $\xi$ on $P$.

In the particular case that $\xi = fG\rd\Phi_a$, the
integral over $\del N_a(t)\cap \del P_t^+$ vanishes for the usual
reason that $G\nu=0$ on any facet of $\del P^+$ with conormal $\nu$
(cf.\ \eqref{e22.7.9.12}). 
Thus we obtain
\begin{equation}
\int_{N_a(t)} \div(fG\rd\Phi_a)\,\rd \sigma_a =
\frac{\rd}{\rd t}\int_{N_a(t)} |\rd\Phi_a|^2_g\,\rd \sigma_a
- \sum_b \int_{N_a(t)\cap N_b(t)} \langle f\langle \rd\Phi_a,
\rd(\Phi_b-\Phi_a)\rangle_g\,\rd \tau_{ab}.
\end{equation}
Summing over $a$, we get the formula
\begin{equation}
\int_{N(t)} \div(fG\rd\Phi)\,\rd \sigma =
\frac{\rd}{\rd t}\int_{N(t)} |\rd\Phi|^2_g\,\rd \sigma
+ \sum_{a<b} \int_{N_a(t)\cap N_b(t)}  f|\rd(\Phi_b-\Phi_a)|^2_g\,\rd \tau_{ab}.
\end{equation}
The last term on the RHS here is $\int_{N(t)}f\,\rd p$ by definition, so
by combining this equation with \eqref{e72.4.9.12}, we obtain
\begin{equation}\label{e1.30.8.12}
\int_{P(t)} \del_i\del_j(fG^{ij})\,\rd x
= - \frac{\rd}{\rd t} \int_{\partial P(t)} f|\rd\Phi|^2_g\,\rd \sigma
- 2\int_{\partial P_l} f\,\rd \sigma - \int_{N(t)} f\,\rd p.
\end{equation}
Substitution of this into \eqref{e35.5.9.12} gives
\begin{equation}\label{e36.6.9.12}
\langle \hat{\rho}_{tk},f\rangle =
k^n\int_{P(t)} f\,\rd x + \frac{1}{2}k^{n-1}\int_{P(t)} s
\,\rd x + \frac{1}{2}k^{n-1}\langle \hat{a}_t,f\rangle + k^{n-2}E_k[f],
\end{equation}
where
\begin{equation}\label{e37.9.9.12}
\langle \hat{a}_t,f\rangle = \int_{N(t)} f\,\rd \sigma - \frac{1}{2}\frac{\rd}{\rd t}\int_{N(t)}
f|\rd\Phi|^2_g\,\rd \sigma - \int_{N(t)} f\,\rd p
\end{equation}
as required.
\end{proof}

Setting $f=1$, we obtain the analogue of Proposition~\ref{p1.9.9.12}
in this case:

\begin{prop}\label{p2.9.9.12}
Let $g$ be any (smooth) toric K\"ahler metric on $X$ and let $s = s(g)$ be the scalar curvature. Denote by $\rd p_t$ the above measure.
If $t$ is not a critical value of the family $\{P(t)\}$,
\begin{equation}\label{e4.27.2.12}
\Vol(\partial P(t)^+) = \int_{P(t)} s - \frac{\rd}{\rd t}\int_{N(t)}
|\nu_t|^2\,\rd \sigma_t - \int_{N(t)} \rd p_t,
\end{equation}
where $s$ is the scalar curvature and $\nu_t$ is the conormal to $N_t$.
\end{prop}

\subsection{Test configurations and K-stability}

In \cite{Don_ToricScal}, Donaldson proposed a definition of
K-stability for polarized varieties $(X,L)$ and, in the toric case,
related K-stability to boundedness properties of the Mabuchi
energy. We shall not reproduce the exact definition here.  The rough
idea is to consider `degenerations' of $(X,L)$ to a (possibly very singular)
polarized variety $(X_0,L_0)$ with a $\bC^\times$-action. In this situation
one can define the Donaldson--Futaki invariant $F_1$ of $(X_0,L_0)$; then $(X,L)$
is K-stable if $F_1 <0$ for all possible degenerations.  

In the toric case, there is a subclass of toric degenerations which
can be defined combinatorially as follows. Let the polarized toric
variety $(X,L)$ correspond to the convex integral polytope $P$.
Now, given the data of the previous section, define
$\tPhi_a(x,t) = \Phi_a(x) -t$ and consider the polytope
$\Gamma\subset \bR^{n+1}$ (the last variable being $t$),
\begin{equation}\label{e11.21.9.12}
\Gamma = P\times [0,\infty) \cap \bigcap_{a} \{\tPhi(x,t) \geq 0\}.
\end{equation}
It is convenient to augment the defining equations for $\Gamma$ by
explicitly including $\tPhi_0(x,t) = t$ which defines the base $t=0$ of $\Gamma$.

We refer to any such $\Gamma$ as (the polytope corresponding to) a
{\em toric test configuration} for $(X,L)$.  We note that the `roof'
of $\Gamma$ (see Figure 2(a)) is a union of $n$-dimensional convex
polytopes $\tilde{N}_a$. Then $X_0$ in this case is obtained by
gluing together the toric varieties corresponding to the $\tilde{N}_a$
to obtain a singular variety.

By definition, a {\em product configuration} arises when $\Gamma$ is
$P\times [0,\infty)$ cut off (possibly obliquely) by a single affine function
$\tPhi(x,t)\geq 0$---see Figure~2(b). A product configuration is
called {\em trivial} if the roof is 
horizontal, i.e. given by $P\times\{c\}$ for some $c>0$.

\begin{figure}[!ht]
    \vspace{0.5cm}
\hspace*{\fill}
\subfloat[][A polytope defining a toric test configuration and a
       non-critical level set of $P_{t}$.]{
    \includegraphics[width=55mm]{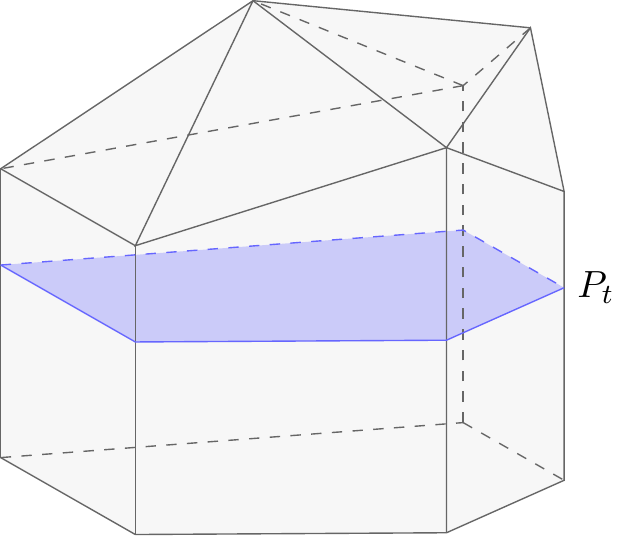}
    \label{fig:1a}
}
\hfill
\subfloat[][A polytope corresponding to a product configuration.]{
    \includegraphics[width=55mm]{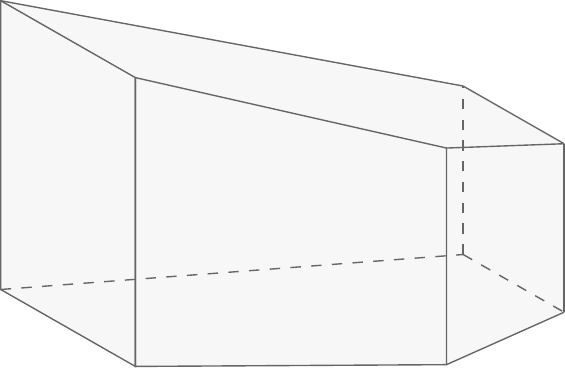}
    \label{fig:1b}
}
\hspace*{\fill}
\label{figures1ab}
\caption{}
\end{figure}

In \cite{Don_ToricScal}, the following combinatorial description of
the Donaldson--Futaki invariant was given:
\begin{prop}
Let $(X,L)$ be a toric variety with moment polytope $P$ and let
$\Gamma$ be a polytope defining a toric test configuration for
$(X,L)$.   Then the Donaldson--Futaki invariant of $\Gamma$ is the coefficient of
$k^{-1}$ in the asymptotic expansion of $w_k/kd_k$, where
\begin{equation}\label{e2.22.2.12}
w_k = N(k\Gamma) - N(kP),\;\; d_k = N(kP)
\end{equation}
and $N(A)$ denotes the number of lattice points in the integral polytope $A$.
\end{prop}

The significance of this definition in relation to K-stability is as 
follows:
\begin{dfn}\label{d1.6.10.12}
Let $(X,L)$ be a toric variety with moment polytope $P$. 
We say that $(X,L)$ is K-polystable with respect to toric test
configurations if for every toric test configuration corresponding to
an $(n+1)$-dimensional polytope $\Gamma$ as above, the
Donaldson--Futaki invariant $F_1$ is $\leq 0$, with equality if and
only if $\Gamma$ corresponds to a product test configuration.
\end{dfn}

The reader is referred to \cite[\S4.2]{Don_ToricScal} for the details.

Our next theorem gives a formula for the Donaldson--Futaki invariant,
given a choice of toric K\"ahler metric $g$, which is very analogous
to the formula for the slope given in Theorem~\ref{thm_3}.

Clearly $\Gamma$ is closely related to the family $P_t$ of
polytopes \eqref{e1.21a.9.12}. More precisely, let
$\pi$ denote the restriction to $\Gamma$ of the projection $(x,t)\mapsto t$.
Then $\pi^{-1}(t) = P_t$, where $P_t$ is as in
\eqref{e1.21a.9.12}.   By analogy with 
\S\ref{s_mov}, let us introduce the following notation:
\begin{itemize}
\item Write $\partial \Gamma = \tilde{N} \cup \partial \Gamma^+$,
  where $N$ is the `roof' of $\Gamma$, that is the union of the facets
  defined by the hyperplanes $\tPhi_a(x,t)=0$ and $\Gamma^+$ is the
  `vertical part' of $\partial \Gamma$---the union of facets contained
  in sets of the form $F\times \bR$, where $F$ is a facet of $P$.

\item Define a positive measure $\rd \tilde{p}$ with support on the
  $(n-1)$-skeleton of $\tilde{N}$ as follows: if $\tilde{N}_{ab} =
  \Gamma\cap \{\tPhi_a = 0\}\cap \{\tPhi_b =0\}$, define
\begin{equation}\label{e1.6.10.12}
\rd\tilde{p}_{ab} = |\rd \tPhi_a-\rd \tPhi_b|^2_g\,\rd
\tilde{\tau}_{ab} \mbox{ on the interior of }\tN_{ab}.
\end{equation}
This makes sense because $\rd\tPhi_a - \rd\tPhi_b = \rd\Phi_a - \rd
\Phi_b$ is a $1$-form on $P$; its length can thus be measured with the 
toric metric $g$.
Then define $\rd\tilde{p}$ on $\tilde{N}$ to be equal to
$\rd\tilde{p}_{ab}$ on the relative interior of $\tilde{N}_{ab}$ for
all $a\neq b$.
\item We say that $c$ is a critical value of $\pi$ if $\pi^{-1}(c)$ contains a
  vertex of $\Gamma$.  It is easily seen that $c$ is a critical value
  of $\pi$ if and only if it is one of the $c_j$ of
  \eqref{e12.5.10.12}. 
\end{itemize}

Now define
\begin{equation}\label{e11.13.6.12}
\Delta(\Gamma) = \frac{1}{\Vol(\Gamma)} \int_{\tN} \rd \widetilde{p},
\end{equation}
the integral being over the roof $\tN$ of $\Gamma$.  Since $\rd
\widetilde{p}$ is a non-negative measure,
\begin{equation}\label{e21.16.6.12}
\Delta(\Gamma) \geq 0\mbox{ for any }\Gamma
\end{equation}
and
\begin{equation}\label{e22.16.6.12}
\Delta(\Gamma) = 0 \mbox{ if and only if the roof of $\Gamma$ has no codimension-2 faces}.
\end{equation}
In other words, $\Delta(\Gamma)\geq 0$ with equality if and only if $\Gamma$ corresponds to a product test configuration.

\begin{thm}
Let $(X,L)$ be a smooth polarized toric variety with moment polytope
$P\subset \bR^n$.  Let $\Gamma \subset \bR^{n+1}$ be a polytope
defining a toric test configuration for $(X,L)$.  Then, for any choice
of toric K\"ahler metric $g$ in the K\"ahler class $c_1(L)$ on $X$,
the Donaldson--Futaki invariant of $\Gamma$ is
given by 
\begin{equation}\label{e4.25.2.12}
F_1 = \frac{\Vol(\Gamma)}{2\Vol(P)}\left(\Av_\Gamma(\pr_1^*(s(g)) - \Av_P(s(g))
- \Delta(\Gamma)\right),
\end{equation}
where $s(g)$ is the scalar curvature of $g$,  $\Av_A(f)$ denotes average value of the function $f$ over the set $A$ and $\pr_1$ denotes the vertical projection $\Gamma\to P$.
\label{t1.25.2.12}\end{thm}

The following is a simple consequence of \eqref{e4.25.2.12}:
\begin{cor}[cf.\ \cite{MR2407096}]
Suppose that $X$ admits a toric cscK metric in the K\"ahler class
$c_1(L)$.  Then the Donaldson--Futaki invariant of any toric test
configuration with polytope $\Gamma$ is $\leq 0$, with equality if and
only if $\Gamma$ is a product configuration. In other words the
existence of a toric cscK metric implies that $(X,L)$ is K-polystable
with respect to toric test configurations. 
\end{cor}
\begin{rmk} We follow the sign convention of \cite{Don_ToricScal}
  rather than \cite{Thomas_cscK} here, so for us negative
  Donaldson--Futaki invariant corresponds to stability. 
\end{rmk}

\begin{proof} (Of the corollary.) 
If the scalar curvature $s$ is constant, then it is equal to its
average over $P$ and also to the average of $\pr_1^*(s)$ over
$\Gamma$. Thus these averages cancel from \eqref{e4.25.2.12}, leaving 
$$
F = -\frac{\Vol(\Gamma)}{2\Vol(P)}\Delta(\Gamma).
$$
The result now follows from \eqref{e21.16.6.12} and \eqref{e22.16.6.12}.
\end{proof}

\subsection{Computation of the Donaldson--Futaki invariant}

We use the following observation:
\begin{lem}[\cite{Don_ToricScal}]  The large-$k$ expansion of $w_k$ is given
  by
\begin{equation}\label{e1.25.2.12}
w_k = \Vol(\Gamma)k^{n+1} +\frac{1}{2}\Vol(\partial \Gamma^+) k^{n} +
O(k^{n-1}).
\end{equation}
\end{lem}

Given this result, the main problem is to understand $\Vol(\partial
\Gamma^+)$ in terms related to the metric.  Since the intersection of
$\del\Gamma_+$ with a horizontal slice $P_t$ is the `old part' $\del
P_+$ of the boundary of $P_t$, we have
\begin{equation}\label{e1.14.6.12}
\Vol(\partial \Gamma^+) = \int_0^{c_m} \Vol(\partial P(t)^+)\,\rd t.
\end{equation}
where $c_m$ is the largest critical value of $\pi$ as in
\eqref{e12.5.10.12}.  Combining this with Proposition~\ref{p2.9.9.12},
we shall obtain the formula:
\begin{prop}\label{p1.6.10.12}
With the above definitions and notation, we have
\begin{equation}\label{e31.6.10.12}
\Vol(\partial \Gamma^+) = \int_\Gamma \pr_1^*(s) - \Vol(\Gamma)\Delta(\Gamma),
\end{equation}
where $\pr_1$ is the restriction to $\Gamma$ of the projection
$(x,t)\mapsto x$.
\end{prop}
\begin{proof}
We integrate \eqref{e4.27.2.12} over each interval $(c_{j-1},c_j)$ and
sum over $j$, getting
\begin{align} \Vol(\partial \Gamma^+) &= \int_\Gamma \pr_1^*(s) + \int_{N(0)}|\nu|^2\,\rd \sigma -
\sum_{j=1}^{m-1} \left(\int_{N(t_j^-)}|\nu|^2 \,\rd \sigma  -
  \int_{N(t_j^+)}|\nu|^2 \,\rd \sigma  \right)\nonumber
  \\
  &- \int_{N(t_m^-)} |\nu|^2 \,\rd \sigma
- \int_0^{t_m} \rd p_t\,\rd t.
\label{e2.25.2.12}
\end{align}
Now the integral over $N_0$ is zero because---by definition---$N_0$ is
empty.    Thus we have
\begin{equation}\label{e1.16.6.12}
\Vol(\partial \Gamma^+) - \int_\Gamma \pr_1^*(s)  = -
\sum_{j=1}^{m-1} \left(\int_{N(t_j^-)}|\nu|^2 \,\rd \sigma  -
  \int_{N(t_j^+)}|\nu|^2 \,\rd \sigma  \right)
- \int_{N(t_m^-)} |\nu|^2 \,\rd \sigma
- \int_0^{t_m} \rd p_t\,\rd t.
\end{equation}
The next lemma matches up the terms on the RHS of this equation
with the terms in the sum defining 
$\Delta(\Gamma)$, cf. \eqref{e11.13.6.12}.  Recall that the roof
$\tilde{N}$ of $\Gamma$ is a union of facets $\tilde{N}_a$ and its
$(n-1)$-skeleton consists of intersections of the form $\tilde{N}_{ab}
= \tilde{N}_a \cap \tilde{N}_b$.  We call $\tilde{N}_{ab}$ {\em
  horizontal} if it is contained in a horizontal slice $P_t$ and {\em
  non-horizontal} otherwise.  The point is that if $\tilde{N}_{ab}$ is
horizontal, then it has to appear as a facet contained in $N(t)
\subset \del P(t)$ and moreover $t$ has to be a critical value of
$\pi$, since $\tilde{N}_{ab} \subset \del P_t$ certainly implies that
$P_t$ contains a vertex of $\Gamma$. On the other hand, if
$\tilde{N}_{ab}$ is not horizontal, then it meets $P_t$ non-trivially
for $t$ in some interval $I$, and for each $t$ in the interior of $I$,
$P_t\cap \tilde{N}_{ab} = N_{ab}(t)$ is part of the $(n-2)$-skeleton
of $N(t)$.  These rather simple observations may nonetheless help with
the following:

\begin{figure}[!ht]
    \vspace{0.5cm}
\hspace*{\fill}
\subfloat[A critical slice $P_c$ in the case where $\rd\sigma_t$ and $\nu_t$ vary continuously through $t=c$.]{
    \label{fig2a}
    \includegraphics[width=55mm]{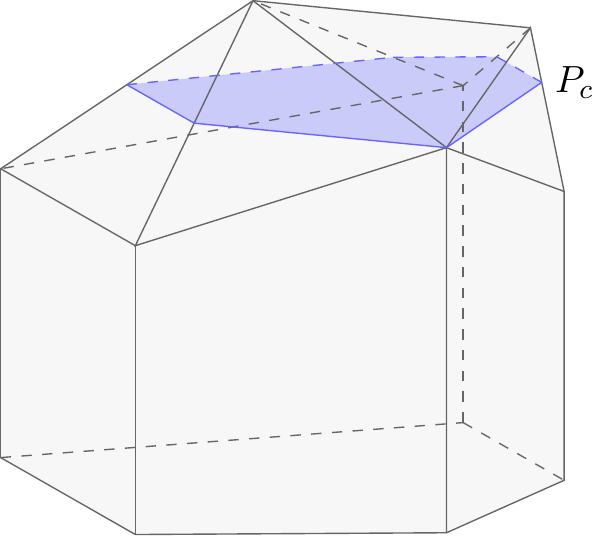}
}
\hfill
\subfloat[A critical slice $P_c$ in the case where $\rd\sigma_t$ and $\nu_t$
    do not vary continuously through $t=c$.
]{
    \label{fig2b}
    \includegraphics[width=55mm]{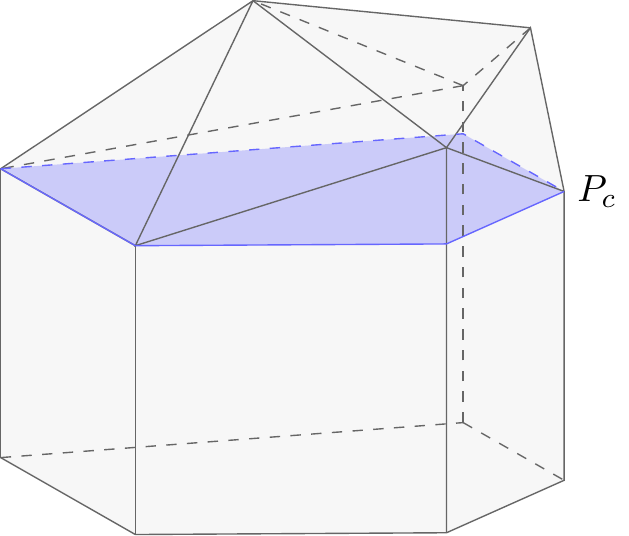}
}
\hspace*{\fill}
\label{figures2}
\caption{}
\end{figure}

\begin{lem}\label{l1.16.6.12}
We have
\begin{equation}\label{e2.16.6.12}
\sum_{j=1}^{m-1} \left(\int_{N(c_j^-)}|\rd\Phi|_g^2 \,\rd \sigma 
-  \int_{N(c_j^+)}|\rd\Phi|_g^2 \,\rd \sigma  \right)
+ \int_{N(c_m^-)} |\rd\Phi|_g^2 \,\rd \sigma =
\sum_{\tilde{N}_{ab} \, \mathrm{ horizontal}}
\int_{\tilde{N}_{ab}} \rd\tilde{p}_{ab}.
\end{equation}
and
\begin{equation}\label{e11.6.10.12}
\int_0^{c_m} \rd p_t\,\rd t = \sum_{\tilde{N}_{ab}\, \mathrm{ not\,
    horizontal} } \int_{\tilde{N}_{ab}}\rd\tilde{p}_{ab}.
\end{equation}
\end{lem}

\begin{proof} (See Fig.\ 3).  Pick $c_j < c_m$, consider any 
facet $F$, say, of $P_{c_j}$ and in particular the contribution $F$
makes to the sum on the LHS of \eqref{e2.16.6.12}.  Suppose first that
$F$ is the intersection of
$P_{c_j}$ with a single facet $\tilde{N}_a$ of $\Gamma$. Then 
$F_t = \tilde{N}_a\cap P_t$ for $t$ near $c_j$.    Hence the Leray
form and conormal of $F_t$ vary continously for $t$ near $c_j$. So
such facets contribute nothing to the sum on the LHS of \eqref{e2.16.6.12}.

The other case to consider is that 
$F = \tilde{N}_a\cap \tilde{N}_b$.  Note that this
is necessarily a horizontal face of $\tilde{N}$. Suppose that $a$ and
$b$ are ordered so that $F_t = P_t\cap \tilde{N}_a$ for $t < c_j$ but
$F_t = P_t \cap \tilde{N}_b$ for $t > c_j$ (it always being assumed
that $|t-c_j|$ is small).  Then the contribution to the sum on the LHS
of \eqref{e2.16.6.12} is 
\begin{equation}\label{e3.16.6.12}
\int_{F}( |\rd\Phi_{a}|_g^2\,\rd \sigma_a - |\rd\Phi_{b}|_g^2\,\rd
\sigma_b).
\end{equation}
We claim that this is equal to
\begin{equation}\label{e4.16.6.12}
\int_{\tilde{N}_{a}\cap \tilde{N}_b}|\tPhi_a -
\tPhi_b|_g^2\,\rd\tilde{\tau}_{ab} = \int_{\tN_{ab}} \rd\tilde{p}_{ab}.
\end{equation}
Because $\tilde{N}_a$ and $\tilde{N}_b$ meet in a horizontal plane, we can choose affine coordinates on $\bR^n$ so that
\begin{equation}
\tPhi_a(x,t) = Ax_1-t,\;\; \tPhi_b(x,t) = Bx_1-t , \mbox{ with }A>B >0.
\end{equation}
(Thus $\Gamma$ is given locally by the intersection of the half-spaces
$Ax_1\geq t$ and $B x_1 \geq t$.)   Then
\begin{equation}
\rd\Phi_a = -A\rd x_1, \rd\Phi_b = - B\rd x_1,\;\;\rd\sigma_a = \frac{1}{A}\rd x_2\ldots\rd x_n,
\rd\sigma_b = \frac{1}{B}\rd x_2\ldots\rd x_n,
\end{equation}
and so the integrand in \eqref{e3.16.6.12} is
\begin{equation}
(A-B)|\rd x_1|_g^2\,\rd x_2\ldots\rd x_n.
\end{equation}
On the other hand,
\begin{equation}
|\rd\tPhi_a-\rd\tPhi_b|^2 = (A-B)^2|\rd x_1|^2_g,\;\;
\rd\tilde{\tau}_{ab} = \frac{1}{A-B}\rd x_2\ldots \rd x_n,
\end{equation}
from which the equality of \eqref{e3.16.6.12} and \eqref{e4.16.6.12} follows.

This proves \eqref{e2.16.6.12} except that we have ignored the last term on the LHS and the codimension-2 faces of $\Gamma$ contained in $t=t_m$.   The required equality can
be obtained by straightforward modification of the above discussion and further details are omitted.
\end{proof}

The proof of Proposition~\ref{p1.6.10.12}  follows by combining
this Lemma with the definition of $\Delta(\Gamma)$.
\end{proof}

We can now easily complete the proof of Theorem~\ref{t1.25.2.12}: just
substitute \eqref{e31.6.10.12} into \eqref{e1.25.2.12} and use
\begin{equation}\label{e32.6.10.12}
kd_k = \Vol(P)k^{n+1}\left( 1 + \frac{1}{2}\Av_P(s)k^{-1} + O(k^{-2})\right)
\end{equation}
to calculate
$$
\frac{w_k}{kd_k} = \frac{\Vol(\Gamma)}{\Vol(P)} + 
\frac{\Vol(\Gamma)}{2\Vol(P)}\left(\Av_\Gamma(\pr_1^*(s)) - \Av_P(s) - \Delta(\Gamma)\right)k^{-1} + O(k^{-2}).
$$
The result now follows from the definition of $F_1$ as the coefficient
of $k^{-1}$ in $w_k/kd_k$.

\bibliographystyle{amsalpha}
\bibliography{paper}

\end{document}